\documentclass{article}
\usepackage[utf8]{inputenc}
\usepackage{amsmath,amssymb,amsthm,bm}
\usepackage{amsfonts}
\usepackage{mathrsfs}
\usepackage{graphicx}
\usepackage{epstopdf}
\usepackage{enumitem}
\usepackage{overpic}
\usepackage{multirow}
\usepackage{nicefrac}
\numberwithin{equation}{section}
\usepackage{float}
\usepackage{caption}
\usepackage{multicol}
\usepackage{dsfont}
\usepackage{bbm}
\usepackage{bm}
\usepackage{url}
\usepackage{verbatim}
\usepackage{algorithm,algorithmic}
\usepackage{xcolor}
\usepackage{arydshln}
\usepackage{subfig}
\usepackage{nicematrix}

\usepackage[colorlinks,linkcolor=blue,
            citecolor=blue]{hyperref}

\usepackage{geometry}
\usepackage{cleveref}
\usepackage{slashed}

\newcommand{\diag}{\mathrm{diag}}

\newcommand{\argmin}{\operatornamewithlimits{argmin}}

\newcommand{\trace}{\mathrm{trace}}

\newcommand{\Det}{{\rm Det}}

\newtheorem{theorem}{Theorem}

\newtheorem{definition}{Definition}

\newtheorem{lemma}{Lemma}

\newtheorem{remark}{Remark}

\newtheorem{property}{Property}

\providecommand{\keywords}[1]
{
  \small	
  \textbf{\textit{Key words---}} #1
}

\graphicspath{
{images/}
}

\title{$n$-Dimensional Volumetric Stretch Energy Minimization for Volume-/Mass-Preserving Parameterizations}
\author{Zhong-Heng Tan, Tiexiang Li, Wen-Wei Lin, Shing-Tung Yau}
\date{}

\begin{document}
\maketitle

\begin{abstract}
    In this paper, we develop an $n$ dimensional volumetric stretch energy ($n$-VSE) functional for the volume-/mass-preserving parameterization of the $n$-manifolds topologically equivalent to $n$-ball. The $n$-VSE has a lower bound and equal to it if and only if the map is volume-/mass-preserving. This motivates us to minimize the $n$-VSE to achieve the ideal volume-/mass-preserving parameterization. In the discrete case, we also guarantee the relation between the lower bound and the volume-/mass-preservation, and propose the spherical and ball volume-/mass-preserving parameterization algorithms. The numerical experiments indicate the accuracy and robustness of the proposed algorithms. The modified algorithms are applied to the manifold registration and deformation, showing the versatility of $n$-VSE.
\end{abstract}



\section{Introduction}

In this paper, we consider $\mathcal{M}$ as an $n$-dimensional manifold which is topologically equivalent to an $n$-ball $\mathbb{B}^n = \{\mathbf{x}\in \mathbb{R}^n~|~\|\mathbf{x}\|_2^2 \leq 1\}$. An $n$-dimensional parameterization between $\mathcal{M}$ and $\mathbb{B}^n$ can be referred as to find a bijective map from $\mathcal{M}$ to $\mathbb{B}^n$. The mapping induces a canonical coordinate system on the $n$-manifold, which can be used to simplify the process issue arising from digital geometrical patterns, and medical image data. In particular, surface ($2$-manifold) and volume (3-manifold) parameterizations have been widely applied in computer graphics \cite{MGBS21}, image remeshing \cite{KSNL19,CPXG17}, morphing \cite{MHWW17}, registration \cite{KCLM14,RSWZ17,CPGP18} and texture mapping \cite{LBPS02,LMKC13}, etc. In recent years, several numerical algorithms from different approaches for the computation of surface \cite{AGGP21,PTLM15,SHSA00, RSKC17, MHWW19}, or volume \cite{KSWC17,MHTL19,XGFL16,MHTL20,JWTL23,YJGP15} parameterizations have been well-developed and proposed by various research groups. Most of parameterization algorithms using nongrid mapping method for surfaces and 3-manifolds are in view of the minimization of distortions of angle or area for 2-manifolds and of volume/mass for 3-manifolds. The corresponding parameterizations, called angle-preserving (conformal), area-preserving (authalic/equiareal) parameterizations, and volume-/mass-preserving parameterizations, are computed by minimizing the Dirichlet energy \cite{MHWW17,YCWW21}, stretch energy \cite{MHWW19} and volumetric stretch energy \cite{MHTL19,MHTL20,TMWH23}, respectively.

On the other hand, another more general and useful approach, namely, a unifying framework of $n$-dimensional quasi-conformal maps \cite{DPGP22} to minimize quasi-conformal distortion, volume distortion, landmark mismatch and to ensure the bijectivity and volume prior information for the computation of $n$-dimensional quasi-conformal mappings between $\mathcal{M}$ and $\mathbb{B}^n$. More precisely, based on the idea \cite{YTKC15}, the framework \cite{DPGP22} minimizes a generalized conformality distortion to measure the dilation of an $n$-dimensional quasi-conformal maps and satisfies the landmark constraints. Furthermore, the orientation constraint (determinant $>0$) is converted into an equality constraint by using the exponential function with adding a regularizer to minimize the volume change. The associated $n$-dimensional quasi-conformal maps can be solved by ADMM with $O(\frac{1}{k})$ convergence. The framework \cite{DPGP22} can be successfully applied to the adaptive remeshing, graphics and medical image registration.

Based on the techniques of volumetric stretch energy minimization developed in \cite{MHTL19,TMWH23}, in this paper, we first introduce the $n$-dimensional volumetric stretch energy functional on $\mathcal{M}$ and the associated volumetric stretch Laplacian matrix with modified cotangent weights. We propose the $n$-dimensional volumetric stretch energy minimization ($n$-VSEM) for the computation of an approximately $n$-ball volume-/mass-preserving parameterization from $\mathcal{M}$ to $\mathbb{B}^n$. The minimal volumetric stretch energy is shown to have the lower bound $\frac{\nu(\mathbb{B}^n)^2}{\mu(\mathcal{M})}$, where $\mu$ and $\nu$ are the measure defined on $\mathcal{M}$ and $\mathbb{B}^n$, respectively, and the minimizer forms a volume-/mass-preserving map if and only if the volumetric stretch energy attains this lower bound. This conclusion has been proved in \cite{TMWH23} when $n = 3$ as a special case with $\nu(\cdot) = |\cdot|$.
However, the proof used a too strong constraint $|f(\mathcal{M})| = \mu(\mathcal{M})$. In this paper, we loose the constraint and give the theorems in continuous and discrete cases. This relationship critically and successfully provides the setting for the modified volumetric stretch energy functional.
In fact, $n$-VSEM with constraints have some differences with the unifying framework for $n$-dimensional quasi-conformal mapping by (i) replacing two distortion and smooth balancing items by a volumetric stretch energy functional with modified cotangent weights; (ii) replacing the ADMM method by the fixed-point type iteration, for which according to numerical experiments, both methods have sublinear convergence rate of $O(\frac{1}{k})$;

The main contribution of this paper is four folds.

\begin{itemize}
    \item We propose a continuous $n$-dimensional volumetric stretch energy ($n$-VSE) functional on $\mathcal{M}$ having the lower bound $\frac{\nu(\mathbb{B}^n)^2}{\mu(\mathcal{M})}$ and show that the minimizer is volume-/mass-preserving if and only if the minimal energy attains $\frac{\nu(\mathbb{B}^n)^2}{\mu(\mathcal{M})}$.

    \item The discrete $n$-VSE can be written as a quadratic form with respect to a Laplacian matrix with modified cotangent weights having local sparse connectivity. The stretch energy defined in \cite{MHWW19,MHTL19,MHTL20,TMWH23} are equivalent to $n$-VSE when $n = 2$ and $n = 3$. A theorem about the lower bound of discrete $n$-VSE is proposed as the motivation for the proposed $n$-VSEM algorithm.

    \item We propose the $n$-VSEM algorithm for the efficient computation of $\varepsilon$-mass-preserving parameterization. The constrained optimization problem induced from the lower bound theorem is divided into a small-scale constrained boundary subproblem and a large-scale unconstrained interior subproblem. The boundary subproblem is solved by the Newton method and the interior subproblem is solved by the fixed-point method.

    \item In the numerical experiments, we shows a volume-preserving parameterization of the $3$ and $4$ dimensional ellipsoids as a special case, which indicates the $n$-VSE can reach the guaranteed lower bound, that is, the resulting map is mass-preserving. For the general manifold, the proposed algorithm compute the $\varepsilon$-mass-preserving maps with low mass distortion, demonstrating the robustness and accuracy. The application of $n$-VSEM on manifold registration and deformation also has a robust performance by the modified algorithms.

\end{itemize}

\section{$n$-dimensional Volumetric Stretch Energy}

Let $\mathcal{M} = \{\mathbf{x} = h(\mathbf{u}), \mathbf{u} = (u_1,u_2,\cdots,u_n) \in \Omega \subseteq \mathbb{R}^n\}$ be a smooth $n$-manifold topologically equivalent to an $n$-dimensional unit ball $\mathbb{B}^n$. Here we suppose that
the map $h$ satisfies that the Jacobian matrix $J_h$ is column full rank ($\det(J_h^\top J_h) \neq 0$), $\mu$ is the measure defined in $\mathcal{M}$ and $\rho_\mu$ be the corresponding density function.
Hence the mass of arbitrary subregion $B \subset \mathcal{M}$ is $\mu(B) = \int_B d\mu = \int_B \rho_\mu d\sigma$.
Let $\mathcal{N}$ be another $n$-manifold topologically equivalent to $\mathcal{M}$ with $\nu$ and $\rho_\nu$ being the measure and density function defined in it. The density functions satisfy $0< L \leq \rho_\mu, \rho_\nu < +\infty$.
An orientation preserving and bijective $C^1$ map $f:\mathcal{M} \to \mathcal{N}$ is said to be mass-preserving if
$\nu(A)/\nu(f(\mathcal{M})) = \mu(f^{-1}(A))/\mu(\mathcal{M})$
for every subregion $A \subset f(\mathcal{M})$. From the mathematical analysis we have
\begin{align}
    \nu(A) &= \int_A d\nu = \int_{(f\circ h)^{-1} (A)} (\rho_\nu \circ f \circ h) \sqrt{\det(J_{f\circ h}^\top J_{f\circ h})} ds,\\
    \mu(f^{-1}(A)) &= \int_{f^{-1} (A)} d\mu = \int_{(f\circ h)^{-1} (A)} (\rho_\mu \circ h) \sqrt{\det(J_{h}^\top J_{h})} ds,
\end{align}
where $ds$ is the volume element on $\Omega$. Therefore, the $C^1$ map $f$ should satisfy
\begin{align} \label{eq:masspres}
    \frac{(\rho_\nu \circ f \circ h) \sqrt{\det(J_{f\circ h}^\top J_{f\circ h})}}{\int_\Omega (\rho_\nu \circ f \circ h) \sqrt{\det(J_{f\circ h}^\top J_{f\circ h})} ds}
    = \frac{(\rho_\mu \circ h)\sqrt{\det(J_{h}^\top J_{h})}}{\int_\Omega (\rho_\mu \circ h)\sqrt{\det(J_{h}^\top J_{h})} ds},
\end{align}
By the way, if $\rho_\mu = \rho_\nu \equiv 1$ on $\mathcal{M}$, $f$ is said to be volume-preserving, which is a special case of the mass-preservation. In the rest of the paper, we only discuss the mass-preserving map.

To find the ideal mass-preserving map $f$, we first define the $n$-dimensional volumetric stretch energy ($n$-VSE) functional of the map $f$ as
\begin{align} \label{def:EVfcontinue}
    E_V(f) = \int_\mathcal{M} \left[\frac{\rho_{\nu} \circ f}{\rho_\mu} \Det(\nabla_\mathcal{M} f)\right]^2 d\mu,
\end{align}
where $\Det(\cdot)$ is the pseudo-determinant, the product of all nonzero eigenvalues of the matrix, and $\nabla_\mathcal{M} f := J_{f\circ h} J_{h}^\dag$ is the tangential gradient of the map $f$ on $\mathcal{M}$. By the property that $\lambda(AB^\top) = \lambda(B^\top A)\cup \{0,\cdots,0\}$ for $A,B \in \mathbb{R}^{m\times n}$, $m \leq n$ with $\lambda$ being the spectrum of a matrix, we have $\Det(AB^\top) = \Det(B^\top A)$. Then for $\det(J_{h}^\top J_h) \neq 0$ and $\det(J_{f\circ h}^\top J_{f\circ h}) \neq 0$,
\begin{align}
    \Det(\nabla_\mathcal{M} f)^2
    =& \Det(\nabla_\mathcal{M} f^\top \nabla_\mathcal{M} f)
    = \Det(J_{f\circ h} (J_{h}^\top J_h)^{-1} J_{f\circ h}^\top) \\
    =& \Det((J_{h}^\top J_h)^{-1} (J_{f\circ h}^\top J_{f\circ h}))
    = \frac{\det(J_{f\circ h}^\top J_{f\circ h})}{\det(J_{h}^\top J_h)}. \label{eq:Detderive}
\end{align}
Hence, the $n$-VSE in \eqref{def:EVfcontinue} can be written as
\begin{align}
    E_V(f) =& \int_\mathcal{M} \left(\frac{ (\rho_{\nu} \circ f )\sqrt{\det(J_{f\circ h}^\top J_{f\circ h})} }{\rho_\mu \sqrt{\det(J_{h}^\top J_h)}} \right)^2 d\mu \label{eq:EVfdetM} \\
    =& \int_\Omega \frac{\left( (\rho_{\nu} \circ f \circ h)\sqrt{\det(J_{f\circ h}^\top J_{f\circ h})} \right)^2}{(\rho_\mu \circ h) \sqrt{\det(J_{h}^\top J_h)}} ds.\label{eq:EVfdet}
\end{align}

\begin{remark}

In \eqref{eq:EVfdetM}, the formula $\sqrt{\det(J_{h}^\top J_h)}$ represents the volume ratio of $h(\tau)$ to $\tau$ with $\tau$ being a infinitesimal region on $\Omega$, that is, $\sqrt{\det(J_{h}^\top J_h)} = \lim_{|\tau| \to 0} \frac{|h(\tau)|}{|\tau|}$. Similarly, we also have $\sqrt{\det(J_{f\circ h}^\top J_{f\circ h})} = \lim_{|\tau| \to 0} \frac{|f\circ h(\tau)|}{|\tau|}$. Hence, from the derivation in \eqref{eq:Detderive}, $ \Det(\nabla_\mathcal{M} f) $ means the volume ratio of the $f\circ h(\tau)$ on $\mathcal{N}$ to $h(\tau)$ on $\mathcal{M}$. In the geometric perspective, letting $\sigma = h(\tau)$, the $n$-VSE is the sum of square of the mass ratios of each image infinitesimal region $f(\sigma)$ to the origin infinitesimal region $\sigma$ in $\mathcal{M}$. In other words,
\begin{align}
    E_V(f) = \lim_{\mu(\sigma) \to 0} \sum_{\sigma \subset \mathcal{M}} \left[ \frac{\nu(f(\sigma))}{\mu(\sigma)} \right]^2 \mu(\sigma)
\end{align}.
\end{remark}

Using the representation \eqref{eq:EVfdet}, we give a significant theorem for the lower bound of the $n$-VSE.

\begin{theorem} \label{thm:Sfcontinue}
Let $\mathcal{M}$ be an $n$-manifold to be topologically equivalent to $\mathcal{N}$, $\mu$ and $\nu$ be the measures defined on $\mathcal{M}$ and $\mathcal{N}$, $\rho_\mu$ and $\rho_\nu$ be the corresponding densities. For $n$-VSE $E_V(f)$ defined in \eqref{def:EVfcontinue}, let
\begin{align}
    S_f = \left\{E_V(f)~|~f:\mathcal{M}\to \mathcal{N} \text{ is an orientation preserving and bijective $C^1$ map with } f(\partial \mathcal{M}) = \partial \mathcal{N} \right\}. \label{set:Sfcontinue}
\end{align}
Then we have
\begin{itemize}
    \item[(i)] $\min S_f \geq \frac{\nu(\mathcal{N})^2}{\mu(\mathcal{M})}$;
    \item[(ii)] $f^* = \argmin S_f = \frac{\nu(\mathcal{N})^2}{\mu(\mathcal{M})}$ if and only if $f^*$ is mass-preserving.
\end{itemize}
\end{theorem}

\begin{proof}
For the representation of the $n$-VSE in \eqref{eq:EVfdet}, we treat $(\rho_{\nu} \circ f \circ h) \sqrt{\det(J_{f\circ h}^\top J_{f\circ h})}$ as the variable $\xi$.
The $n$-VSE becomes
\begin{align}
    E_V(\xi) = \int_\Omega \frac{\xi^2}{(\rho_\mu \circ h) \sqrt{\det(J_{h}^\top J_h)}} ds.
\end{align}
Consider the problem
    \begin{align} \label{opt:Sxicontinue}
        \min_{\xi} S_\xi := \left\{ E_V(\xi) \left| \int_\Omega \xi ds = \nu(\mathcal{N}), \xi \geq 0 \right.\right\}.
    \end{align}
    By Euler-Lagrange equation, we have
    \begin{align} \label{eq:KKTcontinue}
        \frac{2\xi}{(\rho_\mu \circ h) \sqrt{\det(J_{h}^\top J_h)}} + \lambda = 0.
    \end{align}
    It follows that $2\xi = -\lambda (\rho_\mu \circ h) \sqrt{\det(J_{h}^\top J_h)}$. Integrating the both sides and using the equalities $\nu(\mathcal{N}) = \nu(f(\mathcal{M})) = \int_\Omega \xi ds$ and $\mu(\mathcal{M}) = \int_\Omega (\rho_\mu \circ h) \sqrt{\det(J_{h}^\top J_h)} ds$, we have $\lambda = -2\nu(\mathcal{N}) / \mu(\mathcal{M})$. Then substituting this formula into \eqref{eq:KKTcontinue}, it follows that $\xi = \frac{\nu(\mathbb{B}^n) (\rho_\mu \circ h) \sqrt{\det(J_{h}^\top J_h)}}{\mu(\mathcal{M})}$. That is
    \begin{align} \label{eq:masspres_thm}
        \frac{(\rho_\nu \circ f \circ h)\sqrt{\det(J_{f\circ h}^\top J_{f\circ h})}}{\nu(\mathcal{N})}
        = \frac{\xi}{\nu(\mathcal{N})}
        = \frac{(\rho_\mu \circ h) \sqrt{\det(J_{h}^\top J_h)}}{\mu(\mathcal{M})},
    \end{align}
    which means the mass-preservation. Additionally, it is easy to verify that the optimization problem \eqref{opt:Sxicontinue} is convex. Therefore, the solution of the Euler-Lagrange equation is the minimizer of the optimization problem \eqref{opt:Sxicontinue} and hence
    $
        \min_\xi S_\xi = \frac{\nu(\mathcal{N})^2}{\mu(\mathcal{M})}
    $.
    Since $S_f \subseteq S_\xi$, we have
    \begin{align}
        \min_f S_f \geq \min_\xi S_\xi = \frac{\nu(\mathcal{N})^2}{\mu(\mathcal{M})},
    \end{align}
    completing the proof of (i). Meanwhile, \eqref{eq:masspres_thm} implies that $\min_f S_f$ reaches the lower bound $\frac{\nu(\mathcal{N})^2}{\mu(\mathcal{M})}$ if and only if the resulting map $f^*$ is mass-preserving. Hence, (ii) is proved.
\end{proof}

For the parameterization problem, (i) in the theorem \ref{thm:Sfcontinue} gives a lower bound of the $n$-VSE under the constraint $f(\mathcal{M}) = \mathcal{N}$. Meanwhile, (ii) tells that the $n$-VSE reaches this lower bound only when the map is mass-preserving. This motivate us to seek for the ideal mass-preserving parameterization via the $n$-VSE minimization problem.

However, it is worth noting that the existence and uniqueness of such a mass-preserving map from $\mathcal{M}$ to $\mathbb{B}^n$ are not theoretically guaranteed. The mass-preserving map absolutely satisfying the conditions in \eqref{set:Sf} for the given manifold $\mathcal{M}$ does not necessarily exist. Furthermore, even if the mass-preserving map exists, it is not necessarily unique. Taking the unit disk $\mathcal{M} = \mathbb{B}^1$ as instance, we can consider the map $f$ in the polar coordinates $(r_f,\theta_f) = (r_v,\theta_v + kr_v)$, where $k\in \mathbb{R}$ is a parameter, $v = (r_v\cos\theta_v,r_v\sin\theta_v)$ and $f(v) = (r_f\cos\theta_f,r_f\sin\theta_f)$. We can confirm that $f:\mathbb{B}^1 \to \mathbb{B}^1$ and
\begin{align}
    \det(J_f) =
    \det\left( \frac{\partial f}{\partial (r_f,\theta_f)} \right)
    \det\left( \frac{\partial (r_f,\theta_f)}{\partial (r_v,\theta_v)} \right)
    \det\left( \frac{\partial v}{\partial (r_v,\theta_v)} \right)^{-1}
    = \det\left( \frac{\partial (r_f,\theta_f)}{\partial (r_v,\theta_v)} \right) = 1.
\end{align}
Hence, the constructed map $f$ is an area-preserving map and we can select different parameters $k$ to obtain different area-preserving maps. As shown in \Cref{fig:D2D}, we plots the cases $k = 0,1,3$. When $k = 0$, $f$ is identity map. When $k = 1,3$, the red lines in \Cref{subfig:D2Dkeq0} become the curves in \Cref{subfig:D2Dkeq1} and \Cref{subfig:D2Dkeq3}, while the areas of the parts formed by the curves remain $\pi/6$.
\begin{figure}[htp]
\centering
\hfill
\subfloat[$k = 0$]{\label{subfig:D2Dkeq0}
\includegraphics[clip,trim = {3.1cm 1cm 2.5cm 0.7cm},width = 0.2\textwidth]{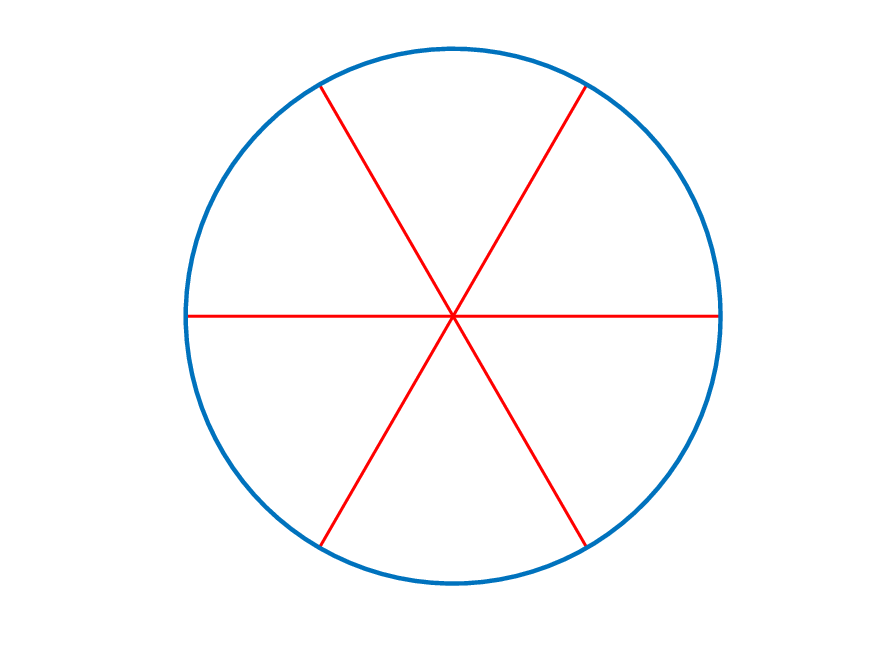}
} \hfill
\subfloat[$k = 1$]{\label{subfig:D2Dkeq1}
\includegraphics[clip,trim = {3.1cm 1cm 2.5cm 0.7cm},width = 0.2\textwidth]{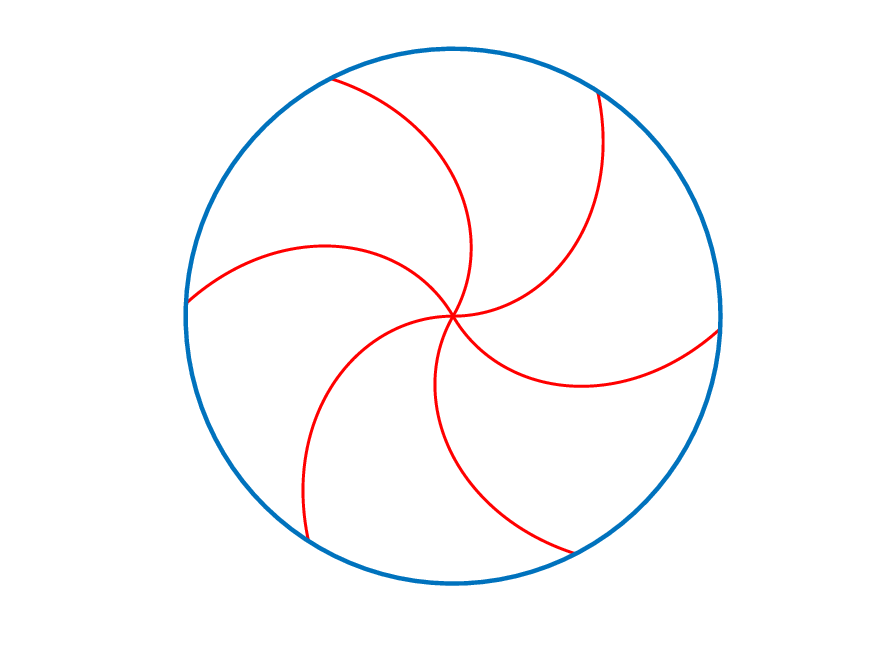}
} \hfill
\subfloat[$k = 3$]{\label{subfig:D2Dkeq3}
\includegraphics[clip,trim = {3.1cm 1cm 2.5cm 0.7cm},width = 0.2\textwidth]{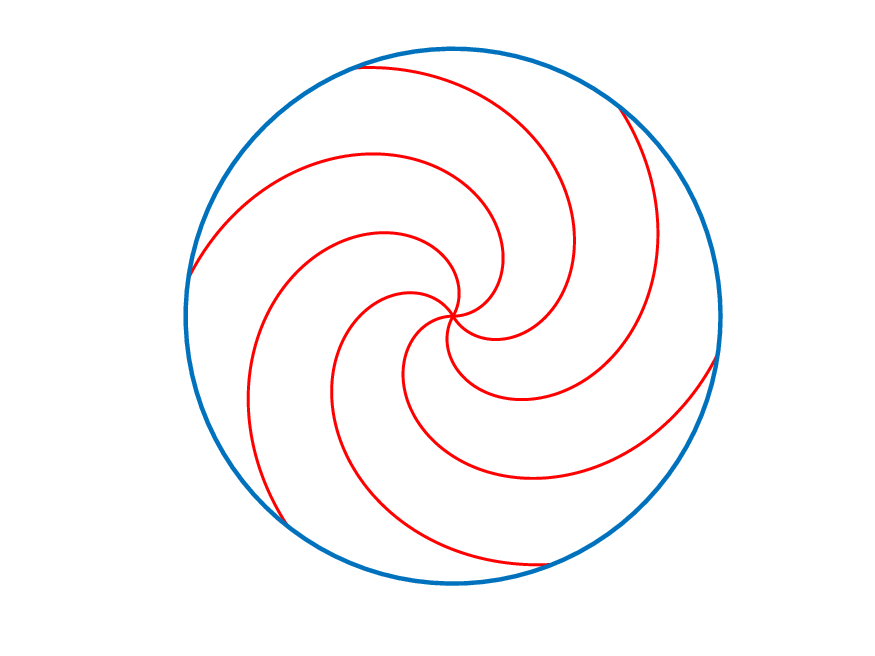}
}
\hfill
\caption{The area-preserving map $(r_f,\theta_f) = (r_v,\theta_v + kr_v)$ with the parameter $k = 0,1,3$.}
\label{fig:D2D}
\end{figure}

Practically, we cannot expect an ideally mass-preserving and unique map as above mentioned. Hence, we take a step back and aim to find a mass-preserving as possible map for the given $n$-manifold $\mathcal{M}$. To measure the degree of mass-preservation, we define the mass-preserving error of a map $f$ as
\begin{align}
    \varepsilon = E_V(f) - \frac{\nu(\mathbb{B}^n)^2}{\mu(\mathcal{M})}
\end{align}
and find an approximately mass-preserving map, named $\varepsilon$-mass-preserving map.
\begin{definition}
    A map $f:\mathcal{M} \to \mathbb{B}^n$ is an $\varepsilon$-mass-preserving map if $\varepsilon = E_V(f) - \frac{\nu(\mathbb{B}^n)^2}{\mu(\mathcal{M})}$.
\end{definition}
Then, we further give the following theorem to show the measurement of the mass-preserving error $\varepsilon$.

\begin{theorem} \label{thm:epsilon_continue}
    Let $f:\mathcal{M}\to \mathbb{B}^n$ be an orientation preserving and bijective $C^1$ map with $f(\partial \mathcal{M}) = \mathbb{S}^{n-1}$ and $\varepsilon = E_V(f) - \frac{\nu(\mathbb{B}^n)^2}{\mu(\mathcal{M})}$ and
    \begin{align}
        \delta = \frac{\nicefrac{(\rho_\nu \circ f)}{\nu(\mathbb{B}^n)}}{\nicefrac{\rho_\mu}{\mu(\mathcal{M})}} \Det(\nabla_{\mathcal{M}} f) - 1.
    \end{align}
    Then
    \begin{align}
        \varepsilon = \frac{\nu(\mathbb{B}^n)^2}{\mu(\mathcal{M})^2} \|\delta\|_{L^2(\mathcal{M})}^2. \label{ineq:epsilon_continue}
    \end{align}
\end{theorem}

\begin{proof}
Let
\begin{align}
&\tilde{\xi} = \frac{(\rho_\nu \circ f \circ h) \sqrt{\det(J_{f\circ h}^\top J_{f\circ h})}}{\nu(\mathbb{B}^n)},\quad
\tilde{\mu} = \frac{(\rho_\mu \circ h) \sqrt{\det(J_{h}^\top J_h)}}{\mu(\mathcal{M})},\\
& \tilde{\mu}_{\text{max}} := \max \tilde{\mu} = \frac{\mu_{\text{max}}}{\mu(\mathcal{M})}, \label{eq:tildemumax}\\
&\tilde{\mu}_{\text{min}} := \min \tilde{\mu} = \frac{\mu_{\text{min}}}{\mu(\mathcal{M})}. \label{eq:tildemumin}
\end{align}
Then $\delta = \tilde{\xi} / \tilde{\mu} - 1$ and
\begin{align}
    \int_\Omega \tilde{\xi} ds = \int_\Omega \tilde{\mu} ds = 1.
\end{align}
Using this equality and \eqref{eq:EVfdet}, we have
\begin{align}
    \varepsilon = & E_V(f) - \frac{\nu(\mathbb{B}^n)^2}{\mu(\mathcal{M})}
    =  \frac{\nu(\mathbb{B}^n)^2}{\mu(\mathcal{M})} \left[ \int_\Omega
    \tilde{\mu}(\delta + 1)^2 ds - 1 \right]\\
    = & \frac{\nu(\mathbb{B}^n)^2}{\mu(\mathcal{M})} \int_\Omega
    \tilde{\mu}[(\delta+1)^2 - 1] ds
    = \frac{\nu(\mathbb{B}^n)^2}{\mu(\mathcal{M})} \int_\Omega
    \tilde{\mu}\delta^2 ds\\
    =& \frac{\nu(\mathbb{B}^n)^2}{\mu(\mathcal{M})^2} \int_\mathcal{M} \delta^2 d\mu = \frac{\nu(\mathbb{B}^n)^2}{\mu(\mathcal{M})^2} \|\delta\|_{L^2(\mathcal{M})}^2.  \label{eq:epsilon_proof1_continue}
\end{align}

\end{proof}

The term $\delta^2$ is the square of the relative error of $\frac{(\rho_\nu \circ f \circ h)\sqrt{\det(J_{f\circ h}^\top J_{f\circ h})}}{\nu(\mathbb{B}^n)}$ and $\frac{(\rho_\mu \circ h) \sqrt{\det(J_{h}^\top J_h)}}{\mu(\mathcal{M})}$, which can represent a local error at a point in $\mathcal{M}$. From \eqref{eq:epsilon_proof1_continue} we can find that $\varepsilon$ is equivalent to the sum of all local errors $\delta^2$ among the manifold $\mathcal{M}$. Hence, it is appropriate to choose $\varepsilon$ as the measurement. In this paper, we aim to find the $\varepsilon$-mass-preserving map with as small error $\varepsilon$ as possible via minimizing $S_f$ in \eqref{set:Sfcontinue}.

\begin{remark}
For $\partial \mathcal{M}$ topologically equivalent to $\mathcal{S}^{n-1}$ and the map $f:\partial \mathcal{M} \to \mathcal{S}^{n-1}$, we can also give the similar theorems to \Cref{thm:Sfcontinue} and \Cref{thm:epsilon_continue}. The theorems and proofs are almost identical and hence not worth being repeated.
\end{remark}

\section{Discrete $n$-Manifold and $n$-Volumetric Stretch Energy}

Now, we focus on the algorithm for the mass-preserving parameterization. At first, we introduce the manifolds in triangular mesh and the $n$-VSE in the discrete perspective. The notations used in the previous section will be abused without being misinterpreted to represent the discrete notations. In this paper, we consider the mass-preserving parameterization to the $n$-ball with $\rho_\nu = 1$. Therefore, in the discrete case, we let $\nu(\cdot) = |\cdot|$ and $\rho := \rho_\mu$ for simplicity.

\subsection{Discrete $n$-Manifold}

The discrete $n$-manifold discussed in this paper is composed of simplices. We first give the definitions of the discrete $n$-manifold and the related $n$-simplex.

\begin{definition}
(1) Let $\{v_0,v_1,\cdots,v_k\}\subseteq \mathbb{R}^n$ with $k\leq n$. A $k$-simplex $\sigma = [v_0,v_1,\cdots,v_k]$ is defined by
    \begin{align}
        \sigma = [v_0,v_1,\cdots,v_k]
        = \left\{\mathbf{x}\in \mathbb{R}^n~\Bigg|~ \mathbf{x} = \sum_{i = 0}^k \alpha_iv_i, \sum_{i = 0}^k \alpha_i = 1, \alpha_i \geq 0 \right\}.
    \end{align}
    An $\ell$-facet $f_a(\sigma)$ of $\sigma$ with $\ell<k$ is defined as $f_a(\sigma) = \{\tau \subseteq \sigma~|~ \tau \text{ is an $\ell$-simplex}\}$.
$v_0,\cdots,v_k$ are called the vertices of the $k$-simplex $\sigma$. Suppose $\tau \subset \sigma$ is also an $\ell$-simplex ($\ell<n$), then $\tau$ is an $\ell$-facet of $\sigma$. Let $f_a(\sigma)$ denote the set of all facets of $\sigma$.

(2) A simplicial $n$-complex $S$ is a union of $n$-simplices such that (i) if an $n$-simplex belongs to $S$, all its facets also belongs to $S$; (ii) if $\sigma_1,\sigma_2\in S$, $\sigma_1\cap \sigma_2 \neq \varnothing$ (empty set), $\sigma_1 \cap \sigma_2 \in f_a(\sigma_1)$ and $f_a(\sigma_2)$.

(3) The link of a vertex $v\in S$ is the union of $(n-1)$-simplex $\tau\in S$ with $v\notin \tau$ and $[v,\tau]$ being an $n$-simplex of $S$.

(4) $v\in S$ is an interior vertex, if the link of $v$ is topologically equivalent to $\mathbb{S}^{n-1}$; otherwise $v$ is a boundary vertex.

\end{definition}

\begin{definition} \label{def:manifold_discrete}
    (1) A discrete $n$-manifold $\mathcal{M}$ is a simplicial $n$-complex which is topologically equivalent to $\mathbb{B}^n$.

    (2) $\partial \mathcal{M}$ is a discrete $(n-1)$-manifold which consists of all $(n-1)$-facets of $\mathcal{M}$ with all vertices being boundary vertices.
\end{definition}

The discrete $n$-manifold has some properties.
\begin{property}
    (1) $\partial \mathcal{M}$ is topologically equivalent to $\mathbb{S}^{n-1}$.

    (2) If $v\in \mathcal{M}$  is an interior vertex, the link of $v$ is topologically equivalent to $\mathbb{S}^{n-1}$. If $v\in \partial \mathcal{M}$ is a boundary vertex of $\mathcal{M}$, the link of $v$ on $\partial \mathcal{M}$ is topologically equivalent to $\mathbb{S}^{n-2}$.
\end{property}

A discrete $n$-manifold $\mathcal{M}$ has $N$ vertices
\begin{align*}
    \mathbb{S}_0(\mathcal{M}) = \{v_i = (v_i^1,v_i^2,\cdots,v_i^n) \in \mathbb{R}^n\}_{i=1}^N
\end{align*}
and the set of $n$-simplex $\mathbb{S}_n(\mathcal{M})$ and the set of $k$-facets $\mathbb{S}_k(\mathcal{M})$, for $k<n$, are defined as follows.
\begin{align}
    \mathbb{S}_n(\mathcal{M}) = \big\{[v_{i_0},v_{i_1},\cdots,v_{i_n}] \subseteq \mathbb{R}^n \text{ for some } v_{i_j} \in \mathbb{S}_0(\mathcal{M}), ~j = 0,1,\cdots,n\big\}.
\end{align}
For $k<n$,
\begin{align}
    \mathbb{S}_k(\mathcal{M}) = \big\{ [v_{i_0},v_{i_1},\cdots,v_{i_k}] \subseteq \mathbb{R}^n ~|~ [v_{i_0},v_{i_1},\cdots,v_{i_k},v_{i_{k+1}}] \in \mathbb{S}_{k+1}(\mathcal{M})
    \text{ with some } v_{i_{k+1}} \in \mathbb{S}_0(\mathcal{M}) \big\}.
\end{align}
For every $\tau \in \mathbb{S}_k(\mathcal{M})$, there is $\sigma \in \mathbb{S}_{p}(\mathcal{M})$ such that $\tau \subset \sigma$ for $k<p \leq n$.

Since an affine map in $\mathbb{R}^n$ is determined by $n+1$ independent point correspondences, a piecewise affine map $f:\mathcal{M}\to \mathbb{R}^n$ on an $n$-simplex mesh $\mathcal{M}$ can be expressed as an $N\times n$ matrix defined by the images $f(v_i)$, where $v_i \in \mathbb{S}_0(\mathcal{M})$, as
\begin{align}
    \mathbf{f} = [\mathbf{f}_1^\top, \mathbf{f}_2^\top, \cdots, \mathbf{f}_N^\top]^\top
    = [\mathbf{f}^1, \mathbf{f}^2, \cdots, \mathbf{f}^n] \in \mathbb{R}^{N\times n},
\end{align}
where $\mathbf{f}^s = [f_1^s,f_2^s,\cdots,f_N^s]^\top \in \mathbb{R}^{N\times 1}$, $s = 1,2,\cdots,n$, and $\mathbf{f}_{t} = f(v_t) = (f^1_t,f^2_t,\cdots,f^n_t) \in \mathbb{R}^{1\times n}$, for $t = 1,2,\cdots,N$. For each point $v\in \mathbb{S}_0(\mathcal{M})$, $v$ must lie into an $n$-simplex $\sigma = [v_0,v_1,\cdots,v_n]$ without loss of generality. The piecewise affine map $f:\mathcal{M} \to \mathbb{R}^n$ can be expressed as a linear combination of $\{\mathbf{f}_i\}_{i=0}^n$ with the barycentric coordinates
\begin{align} \label{eq:bary}
    f\big|_\sigma (v) = \sum_{i=0}^n \alpha_i f(v_i), \quad
    \alpha_i = \frac{1}{|\sigma|} \big| [v_0,\cdots,\hat{v}_i,\cdots,v_n] \big|,
\end{align}
where $\hat{v}_i$ is replaced by $v$, for $i = 0,1,\cdots,n$ and $|\sigma|$ denotes the volume of the $n$-simplex $\sigma$. To prevent folding of $n$-simplex and guarantee the bijectivity via transformation, in this paper, we require the map $f$ to be solved that preserves the orientation, \emph{i.e.} $f(\mathcal{M})$ is a simplicial $n$-complex. In the following discussion, the $n$-manifold is that topologically equivalent to an $n$-ball, and the $(n-1)$-manifold is that topologically equivalent to an $(n-1)$-sphere.

\subsection{Discrete Volumetric Stretch Energy}

For the discrete $n$-manifold $\mathcal{M}$ embedded in  $\mathbb{R}^n $ and the piecewise affine map $f:\mathcal{M} \to \mathbb{B}^n \subset \mathbb{R}^n$, it is easy to verify that the Jacobian matrix $J_{f^{-1}}$ satisfies
\begin{align}
    \det\left(J_{f^{-1}}\big|_{f(\sigma)}\right) := \int_{f(\sigma)} \det\left(J_{f^{-1}}\right) ds = \frac{|\sigma|}{|f(\sigma)|}
\end{align}
for all $\sigma = [v_0,v_1,\cdots,v_n] \in \mathbb{S}_n(\mathcal{M})$. This is the ratio of volumes of the $n$-simplex and its image simplex. Hence, like the continuous case, the piecewise affine map $f:\mathcal{M} \to \mathbb{R}^n$ is said to be induced by $\mathbf{f}$ and is volume-preserving if the Jacobian $J_{f^{-1}}$ satisfies $\det\left(J_{f^{-1}}\big|_{f(\sigma)}\right) = \frac{|\mathcal{M}|}{|f(\mathcal{M})|}$, that is $\nicefrac{|\sigma|}{|\mathcal{M}|} = \nicefrac{|f(\sigma)|}{|f(\mathcal{M})|}$.
Additionally, let the density function $\rho$ defined in $\mathcal{M}$ be piecewise constant and the mass $\mu(\sigma) = \rho(\sigma)|\sigma|$. The map $f:\mathcal{M}\to \mathbb{R}^n$ is mass-preserving if and only if $\rho(\sigma)\det\left(J_{f^{-1}}\big|_{f(\sigma)}\right) = \frac{\mu(\mathcal{M})}{|f(\mathcal{M})|}$, that is, $\nicefrac{\mu(\sigma)}{\mu(\mathcal{M})} = \nicefrac{|f(\sigma)|}{|f(\mathcal{M})|}$.
The Jacobian determinant is an important quantity to measure the mass-preservation. As far as $\rho(\sigma)\det\left(J_{f^{-1}}\big|_{f(\sigma)}\right) = \frac{\mu(\sigma)}{|f(\sigma)|}$ is constant for all $\sigma \in \mathbb{S}_n(\mathcal{M})$, the map $f$ is mass-preserving. In \cite{MHWW19,MHTL19,TMWH23}, it is termed the stretch factor as
\begin{align}
    \sigma_{\mu,f^{-1}}(\sigma) = \rho(\sigma)\det\left(J_{f^{-1}}\big|_{f(\sigma)}\right) =  \frac{\mu(\sigma)}{|f(\sigma)|}, \label{eq:stretchfactor}
\end{align}
and expected to be a constant. Using the stretch factor, \cite{MHWW19,MHTL19,TMWH23} have defined the $2$-dimensional stretch energy and $3$-dimensional stretch energy and given some theorems and properties associated to the stretch energy when $n = 2$ and $3$. Now, we generalize the stretch energy into $n$-dimensional and define the discrete $n$-dimensional volumetric stretch energy functional as
\begin{align} \label{eq:EVcot}
    E_V(f) = \frac{1}{n} \trace \big( \mathbf{f}^\top L_V(f) \mathbf{f} \big) = \frac{1}{n} \sum_{s = 1}^n {\mathbf{f}^{s}}^\top L_V(f) \mathbf{f}^s,
\end{align}
where $L_V(f)$ is a volumetric stretch symmetric Laplacian matrix with
\begin{subequations}
\begin{equation} \label{eq:LVf}
    \big[ L_V(f) \big]_{ij} = \begin{cases}
    -w_{ij}(f), & \text{if } [v_i,v_j]\in \mathbb{S}_1(\mathcal{M}),\\
    \sum_{\ell \neq i} w_{i\ell}(f), & \text{if } i = j,\\
    0, & \text{otherwise}.
    \end{cases}
\end{equation}
\end{subequations}
in which the modified cotangent weights $w_{ij}(f)$ is defined by
\begin{subequations}
\begin{equation} \label{eq:wijf}
    w_{ij}(f) = \frac{1}{n(n-1)} \sum_{[v_i,v_j] \in \mathbb{S}_1(\sigma)} \frac{|f(\sigma_{\hat{i}\hat{j}})|}{\sigma_{\mu,f^{-1}}(\sigma)} \cot\theta_{ij}^{\sigma}(f)
\end{equation}
\end{subequations}
with $\sigma_{\hat{i}\hat{j}}$ being the $(n-2)$-simplex containing the vertices $\mathbb{S}_0(\sigma)\setminus \{v_i,v_j\}$ and
$\theta_{ij}^\sigma(f)$ being the dihedral angle between subspaces $U_i$ and $U_j$ containing $\mathbb{S}_0(f(\sigma))\setminus \{\mathbf{f}_i\}$ and $\mathbb{S}_0(f(\sigma))\setminus \{\mathbf{f}_j\}$ respectively.

In the \Cref{thm:EVf}, we will illustrate that this definition in \eqref{eq:EVcot} is equivalent to our defined $n$-VSE in \eqref{def:EVfcontinue} in the discrete form
\begin{align}
    E_V(f) = \int_\mathcal{M} \frac{\Det(\nabla_\mathcal{M} f)^2}{\rho} ds =  \sum_{\sigma\in \mathbb{S}_n(\mathcal{M}) } \frac{|f(\sigma)|^2}{\mu(\sigma)}.
\end{align}
To proof this theorem, we first generalize the Dirichlet energy to $n$-dimension and give some lemmas associated with it. For the smooth $n$-manifold $\mathcal{M}$ and the orientation preserving and bijective $C^1$ map $f:\mathcal{M}\to \mathbb{R}^n$, the $n$-dimensional Dirichlet energy functional with respect to $f$ is defined as
\begin{align} \label{def:ndDirichlet}
    E_D(f)
    = \frac{1}{n}\int_{\mathcal{M}} \|\nabla_\mathcal{M} f\|_F^2 ds,
\end{align}
In the discrete case, we consider the discrete $n$-manifold $\mathcal{M}$ and the piecewise affine map $f$. The discrete Dirichlet energy is formulated as a quadratic form.

\begin{lemma} \label{lma:Dirichlet}
Suppose $\mathcal{M}$ is a discrete $n$-manifold embedded in $\mathbb{R}^n$ and $f$ defined on $\mathcal{M}$ is a piecewise affine map. Then the $n$-dimensional Dirichlet energy defined in \eqref{def:ndDirichlet} can be formulated as
\begin{align} \label{eq:discreteED}
    E_D(f) = \frac{1}{n}\trace(\mathbf{f}^\top L_D \mathbf{f}),
\end{align}
where
\begin{align} \label{eq:LD}
    \big[L_D\big]_{ij} = \begin{cases}
        -\tilde{w}_{ij}, & \textrm{if } [v_i,v_j] \in \mathbb{S}_1(\mathcal{M}),\\
        \sum_{\ell \neq i} \tilde{w}_{ij}, & \textrm{if } i = j,\\
        0, & \textrm{otherwise.}
    \end{cases}
\end{align}
with the cotangent weight
\begin{align}
    \tilde{w}_{ij} = \frac{1}{n(n-1)}\sum_{\sigma\in \mathbb{S}_n(\mathcal{M}):\{v_i,v_j\}\subset\sigma} |\sigma_{\hat{i}\hat{j}}| \cot\theta_{ij}^\sigma,
\end{align}
in which $\theta_{ij}^{\sigma}$ is the dihedral angle between subspaces containing $\mathbb{S}_0(\sigma)\setminus \{v_i\}$ and $\mathbb{S}_0(\sigma)\setminus \{v_j\}$ respectively.
\end{lemma}

\begin{proof}

Let $q$ be a point in the $n$-simplex $\sigma =[v_0,v_1,\cdots,v_{n}] \in \mathbb{S}_n(\mathcal{M})$. $f(q)$ can be represented by the barycentre coordinates according to \eqref{eq:bary}
\begin{align} \label{eq:baryq}
    f(q) = \sum_{i = 0}^{n} \alpha_i(q) \mathbf{f}_i,\quad
    \alpha_i(q) = \frac{1}{|\sigma|}\big| [v_0,\cdots,v_{i-1},q,v_{i+1},\cdots,v_{n}] \big|.
\end{align}

For the discrete $n$-manifold $\mathcal{M}$ embedded in $\mathbb{R}^n$ and the piecewise affine map $f$, we have $\nabla_\mathcal{M} f = \nabla f$.
In addition, $\nabla f$ is a piecewise constant function and can be also represented by the barycentre coordinates as
\begin{align}
    \nabla f = \sum_{i=0}^n \mathbf{f}_i^\top \nabla \alpha_i.
\end{align}
The volume of the simplex in \eqref{eq:baryq} can be expressed as
\begin{align} \label{eq:VolumeDecom}
    \big| [v_0,\cdots,v_{i-1},q,v_{i+1},\cdots,v_{n}] \big| = \frac{1}{n} \big| \sigma_{\hat{i}} \big| h_i^{\sigma}(q),
\end{align}
where $h_i^{\sigma}(q) = (q - v_\ell)^\top \vec{n}_i$, $\ell \neq i$, is the height corresponding to the face $\sigma_{\hat{i}}$, whose gradient with respect to $q$ is $\vec{n}_i \in \mathbb{R}^{1\times n}$, the inward unit normal vector of the face $\sigma_{\hat{i}}$. Therefore, by the representation in \eqref{eq:VolumeDecom},
we have
\begin{align} \label{eq:NablaLambda}
    \nabla \alpha_i = \frac{\nabla \big| [v_0,\cdots,v_{i-1},q,v_{i+1},\cdots,v_{n}] \big|}{|\sigma|}
    = \frac{1}{n|\sigma|} | \sigma_{\hat{i}} | \vec{n}_i.
\end{align}
Hence,
\begin{align}
    \|\nabla f\|^2_F
    &= \left\|\sum_{i = 0}^n \mathbf{f}_i \nabla \alpha_i \right\|^2_F
    = \left\|\sum_{i = 0}^n \frac{|\sigma_{\hat{i}}|}{n|\sigma|} \mathbf{f}_i^\top \vec{n}_i \right\|^2_F \\
    &= \sum_{i = 0}^n \sum_{j = 0}^n \frac{|\sigma_{\hat{i}}||\sigma_{\hat{j}}|}{n^2|\sigma|^2} \langle \mathbf{f}_{i}^\top\vec{n}_i, \mathbf{f}_{j}^\top\vec{n}_j \rangle
    = \sum_{i = 0}^n\sum_{j = 0}^n \frac{|\sigma_{\hat{i}}||\sigma_{\hat{j}}|}{n^2|\sigma|^2}  (\vec{n}_i \vec{n}_j^\top) (\mathbf{f}_{i}\mathbf{f}_{j}^\top)\\
    &=
    \sum_{i = 0}^n \frac{|\sigma_{\hat{i}}|^2}{n^2|\sigma|^2} \mathbf{f}_{i}\mathbf{f}_{i}^\top
    - \sum_{[v_i,v_j]\in\mathbb{S}_{1}(\sigma)} \frac{|\sigma_{\hat{i}}||\sigma_{\hat{j}}|}{n^2|\sigma|^2} (\mathbf{f}_{i}\mathbf{f}_{j}^\top) \cos\theta_{ij}^\sigma. \label{eq:nfng}
\end{align}
Similar to \eqref{eq:VolumeDecom}, we have $(k-1)|\sigma_{\hat{i}}| = |\sigma_{\hat{i}\hat{j}}| h_j^{\sigma_{\hat{i}}}$, where $h_j^{\sigma_{\hat{i}}}$ is the height corresponding to the face $\sigma_{\hat{i}\hat{j}}$ on the $(n-1)$-simplex $\sigma_{\hat{i}}$. It is easy to verify that $ h_j^{\sigma} = h_j^{\sigma_{\hat{i}}} \sin\theta_{ij}^\sigma$, following that
\begin{align} \label{eq:nfng1}
    &\frac{|\sigma_{\hat{i}}||\sigma_{\hat{j}}|}{n^2|\sigma|^2}
    = \frac{|\sigma_{\hat{i}}|}{n|\sigma|h_j^{\sigma}}
    = \frac{|\sigma_{\hat{i}\hat{j}}|}{n(n-1)|\sigma|\sin\theta_{ij}^\sigma}.
\end{align}
From \cite{Crane19}, we have
\begin{align} \label{eq:nfng2}
|\sigma_{\hat{i}}| = \frac{1}{n-1}\sum_{[v_i,v_j]\in\mathbb{S}_{1}(\sigma)} |\sigma_{\hat{i}\hat{j}}| h_i^{\sigma} \cot\theta_{ij}^\sigma.
\end{align}
Substituting \eqref{eq:nfng1} and \eqref{eq:nfng2} into \eqref{eq:nfng}, we have
\begin{align}
    \|\nabla f\|_F^2 &= \frac{1}{|\sigma|} \sum_{i = 0}^n \sum_{[v_i,v_j]\in \mathbb{S}_1(\sigma)} \frac{1}{n(n-1)} |\sigma_{\hat{i}\hat{j}}| \cot\theta_{ij}^\sigma \mathbf{f}_{ij} \mathbf{f}_j^\top
    = \frac{1}{|\sigma|} \sum_{[v_i,v_j]\in \mathbb{S}_1(\sigma)} \tilde{w}_{ij}^\sigma \mathbf{f}_{ij} \mathbf{f}_i^\top.
\end{align}
As a result,
\begin{align}
    E_D(f) =& \frac{1}{n} \int_{\mathcal{M}}\|\nabla f\|_F^2 ds
    = \frac{1}{n} \sum_{\sigma\in \mathbb{S}_n(\mathcal{M})} \|\nabla f\|_F^2 |\sigma|\\
    =& \frac{1}{n} \sum_{\sigma\in \mathbb{S}_n(\mathcal{M})}\sum_{i\in \mathbb{S}_0(\sigma)} \sum_{[v_i,v_j]\in \mathbb{S}_1(\sigma)} \tilde{w}_{ij}^\sigma \mathbf{f}_{ij} \mathbf{f}_i^\top
    = \frac{1}{n} \sum_{i = 1}^N \sum_{[v_i,v_j]\in \mathbb{S}_1(\mathcal{M})} \tilde{w}_{ij} \mathbf{f}_{ij} \mathbf{f}_i^\top
    = \frac{1}{n} \trace(\mathbf{f}^\top L_D \mathbf{f}).
\end{align}
\end{proof}

One can see that $L_D$ has the identical sparse structure to $L_V(f)$ in \eqref{eq:LVf}. We further let $L_{D}(f)$ be the Laplacian matrix defined on $f(\mathcal{M})$, as like $L_D$ in \eqref{eq:LD}. The cotangent weight of $L_D(f)$ is
$\tilde{w}_{ij}(f) = \frac{1}{n(n-1)} \sum_{\sigma\in \mathbb{S}_n(\mathcal{M}): \{v_i,v_j\} \subset \sigma} |f(\sigma_{\hat{i}\hat{j}})| \cot \theta_{ij}^{\sigma}(f).$
Then, a representation of $|f(\mathcal{M})|$ and its gradient associated with $L_D(f)$ is as in the following lemma.
\begin{lemma} \label{lma:volume}
    The volume of the discrete $n$-manifold $f(\mathcal{M})$ and its gradient with respect to $\mathbf{f}$ can be represented as
    \begin{align}
        &|f(\mathcal{M})| = \frac{1}{n} \trace(\mathbf{f}^\top L_D(f) \mathbf{f}), \label{eq:f(M)cot}\\
        &\nabla_\mathbf{f} |f(\mathcal{M})| = L_D(f) \mathbf{f}. \label{eq:nablaf(M)}
    \end{align}
\end{lemma}

\begin{proof}
    By the definition of the $n$-dimensional Dirichlet energy in \eqref{def:ndDirichlet}, we have $E_D(\emph{id})\big|_\mathcal{M} = |\mathcal{M}|$, where $\emph{id}$ is the identity map. Hence, for the discrete $n$-manifold, by using the discrete formula \eqref{eq:discreteED}, it follows that
    \begin{align} \label{eq:EDid=v}
        |f(\mathcal{M})| = E_D(\emph{id})\big|_{f(\mathcal{M})} := \frac{1}{n} \trace(\mathbf{f}^\top L_D(f) \mathbf{f}).
    \end{align}
    This proves the formula \eqref{eq:f(M)cot}. On the other hand, from the formula (1) in \cite{Crane19}, that is,
\begin{align} \label{eq:citecrane1}
    \nabla_{v_{i}} |\sigma|
    = \frac{1}{n(n-1)}\sum_{[v_i,v_j]\in \mathbb{S}_1(\sigma)} |\sigma_{\hat{i}\hat{j}}| \cot\theta_{ij}^\sigma v_{ij}
    = \sum_{[v_i,v_j]\in \mathbb{S}_1(\sigma)} \tilde{w}_{ij}^\sigma v_{ij},
\end{align}
    we have
    \begin{align} \label{eq:nablafsigma}
        \nabla_{\mathbf{f}_{i}} |f(\sigma)|
        = \sum_{[{v}_i,{v}_j]\in \mathbb{S}_1(\sigma)} \tilde{w}_{ij}^\sigma(f) \mathbf{f}_{ij}
        , ~i = 1,2,\cdots,N.
    \end{align}
    It follows that
    \begin{align}
        \nabla_{\mathbf{f}_{i}} |f(\mathcal{M})| &= \sum_{\sigma\in \mathbb{S}_n(\mathcal{M})} \nabla_{\mathbf{f}_{i}} |f(\sigma)|
        = \sum_{\sigma\in \mathbb{S}_n(\mathcal{M})} \sum_{[v_i,v_j] \in \mathbb{S}_1(\sigma)} \tilde{w}_{ij}^\sigma(f) \mathbf{f}_{ij}
        , ~i = 1,2,\cdots,N.
    \end{align}
    Stacking these equations from $1$ to $N$, \eqref{eq:nablaf(M)} is obtained.
\end{proof}

We can see that the definition of the volumetric stretch energy in \eqref{eq:EVcot} and the volume representation in \eqref{eq:f(M)cot} only differ in the cotangent weights in the Laplacian matrices $L_D(f)$ and $L_V(f)$. More specifically, the edge weights differ in a product of the stretch factor in \eqref{eq:stretchfactor}. That is, $w_{ij}^\sigma(f) = \frac{\tilde{w}_{ij}^\sigma(f)}{\sigma_{\mu,f^{-1}}(\sigma)}$. Using this equality, we can provide a simple representation of the volumetric stretch energy functional in \eqref{eq:EVcot}.

\begin{theorem} \label{thm:EVf}
    The volumetric stretch energy functional in \eqref{eq:EVcot} can be formulated as
    \begin{align} \label{eq:EVv}
        E_V(f) = \sum_{\sigma\in \mathbb{S}_n(\mathcal{M}) } \frac{|f(\sigma)|^2}{\mu(\sigma)}.
    \end{align}
\end{theorem}

\begin{proof}
    Split $\mathcal{M}$ by its $n$-simplex $\{\sigma\}$ and consider the volumetric stretch energy on $\sigma$, then the volumetric stretch energy can also be splitted as
    \begin{align} \label{eq:splitEV}
        E_V(f) = \sum_{\sigma\in \mathbb{S}_n(\mathcal{M})}E_V(f)\big|_\sigma.
    \end{align}
    From \eqref{eq:EDid=v} we have $E_D(\emph{id})\big|_{f(\sigma)} = |f(\sigma)|$. Hence,

    \begin{align}
        E_V(f)\big|_\sigma &= \frac{1}{n}  \sum_{[v_i,v_j]\in \mathbb{S}_1(\sigma)} w_{ij}^\sigma (f)\|\mathbf{f}_{ij}\|^2
        = \frac{1}{n \sigma_{\mu,f^{-1}}(\sigma)}  \sum_{[v_i,v_j]\in \mathbb{S}_1(\sigma)} \tilde{w}_{ij}^\sigma (f)\|\mathbf{f}_{ij}\|^2\\
        &= \frac{|f(\sigma)|}{n\mu(\sigma)} \trace(\mathbf{f}^\top L_D(f) \mathbf{f})
        = \frac{|f(\sigma)|}{\mu(\sigma)} E_D(id)\big|_{f(\sigma)} = \frac{|f(\sigma)|^2}{\mu(\sigma)}. \label{eq:EVsigma}
    \end{align}
    Combining \eqref{eq:splitEV} with \eqref{eq:EVsigma}, \eqref{eq:EVv} is obtained.
\end{proof}

Theorem \ref{thm:EVf} unifies the continuous and discrete $n$-VSE and reveals that the discrete $n$-VSE is formulated by the volumes of the original simplices $\sigma$ and the image simplices $f(\sigma)$.
Next, we propose the discrete case of \Cref{thm:Sfcontinue} for the lower bound of the $n$-VSE, which is the foundation of our method. Different from \Cref{thm:Sfcontinue}, in the discrete case, the volume of the $n$-ball is not a constant. Hence, we must add a constraint that $f(\mathcal{M})$ is a constant.

\begin{theorem} \label{thm:Sf_discrete}
Let $\mathcal{M}\subseteq \mathbb{R}^n$ be a discrete $n$-manifold to be topologically equivalent to $\mathbb{B}^n$ and $\mu(\cdot)$ be the measure defined in $\mathcal{M}$. Let
\begin{align}
    S_f = \left\{E_V(f)~|~f:\mathcal{M}\to \mathbb{B}^n \text{ is orientation preserving and piecewise affine with } \sum_{\sigma\in \mathbb{S}_n(\mathcal{M})} |f(\sigma)| = C, f(\partial \mathcal{M}) = \mathbb{S}^{n-1} \right\}. \label{set:Sf}
\end{align}
Then we have
\begin{itemize}
    \item[(i)] $\min S_f \geq \frac{C^2}{\sum_{\sigma \in \mathbb{S}_n(\mathcal{M})}\mu(\sigma)}$;
    \item[(ii)] $f^* = \argmin S_f = \frac{C^2}{\sum_{\sigma \in \mathbb{S}_n(\mathcal{M})} \mu(\sigma)}$ if and only if $f^*$ is volume-/mass-preserving, i.e., $\frac{|f^*(\sigma)|}{\sum_{\sigma\in \mathbb{S}_n(\mathcal{M})} |f^*(\sigma)|} = \frac{\mu(\sigma)}{\sum_{\sigma\in \mathbb{S}_n(\mathcal{M})} \mu(\sigma)}$, for each $\sigma \in \mathbb{S}_n(\mathcal{M})$.
\end{itemize}
\end{theorem}

\begin{proof}
    Let $m:=\#\big( \mathbb{S}_n(\mathcal{M}) \big)$. Treating the volumes $|f(\sigma)|$ as variables, \eqref{eq:EVv} is temporarily represented as
    \begin{align}
        E_V(\xi) = \sum_{i = 1}^m \frac{\xi_i^2}{\mu(\sigma_i)}, \quad \xi = (\xi_1,\xi_2,\cdots,\xi_m) \in \mathbb{R}^m_+,
    \end{align}
    For the optimization problem
    \begin{align} \label{opt:Sw}
        \min_\xi S_\xi \equiv \left\{
        E_V(\xi) ~\left|~ \sum_{i = 1}^m \xi_i = C,~ \xi_i\geq 0, ~i = 1,2,\cdots,m
        \right.\right\},
    \end{align}
    its Karush–Kuhn–Tucker (KKT) equation is
    \begin{align} \label{eq:solutionKKT}
        \frac{2\xi_i}{\mu(\sigma_i)} + \lambda = 0,\quad i = 1,2,\cdots,m.
    \end{align}
    By the total volume constraints $\sum_{i = 1}^m \xi_i = C$, we have $\lambda = -\frac{2C}{\sum_{i=1}^m \mu(\sigma_i)}$ and $\xi_i = \frac{C \mu(\sigma_i)}{\sum_{i=1}^m \mu(\sigma_i)}$. That is
    \begin{align} \label{eq:vp-KKT}
        \frac{\xi_i}{\sum_{i=1}^m \xi_i} = \frac{\mu(\sigma_i)}{\sum_{i=1}^m \mu(\sigma_i)},\quad i = 1,2,\cdots,m.
    \end{align}
    Additionally, the problem \eqref{opt:Sw} is convex, whose objective function and feasible region are all convex. Therefore, the solution of KKT equation \eqref{eq:solutionKKT} is the minimizer of the optimization problem \eqref{opt:Sw} and hence $\min S_\xi = \frac{C^2}{\sum_{\sigma \in \mathbb{S}_n(\mathcal{M})} \mu(\sigma)}$. Since $\mathcal{S}_{f} \subseteq \mathcal{S}_{\xi}$, (i) $\min S_f \geq \frac{C^2}{\sum_{\sigma \in \mathbb{S}_n(\mathcal{M})} \mu(\sigma)}$ is obtained. Meanwhile, \eqref{eq:vp-KKT} demonstrates the volume-/mass-preserving property. As a consequence, $f^* = \argmin S_f$ attains the lower bound $\frac{C^2}{\sum_{\sigma \in \mathbb{S}_n(\mathcal{M})} \mu(\sigma)}$ if and only if $\frac{|f^*(\sigma)|}{\sum_{\sigma\in \mathbb{S}_n(\mathcal{M})} |f^*(\sigma)|} = \frac{\mu(\sigma)}{\sum_{\sigma\in \mathbb{S}_n(\mathcal{M})} \mu(\sigma)},i = 1,2,\cdots,m$, which leads to (ii).
\end{proof}

Like the \cref{thm:Sfcontinue} in the continuous case, (i) gives the lower bound of the discrete $n$-VSE under the constraints in \eqref{set:Sf} and (ii) indicates that the discrete $n$-VSE reaches the lower bound if and only if the resulting map is mass-preserving. Identically, the existence and uniqueness of the ideal discrete mass-preserving map from $\mathcal{M}$ to a unit $n$-ball $\mathbb{B}^n$ are not theoretically guaranteed. Therefore, the minimum of $E_V(f)$ cannot always reach the lower bound $\frac{C^2}{\sum_{\sigma \in \mathbb{S}_n(\mathcal{M})}\mu(\sigma)}$. Hence, we define the mass-preserving error of the map $f$ as
\begin{align} \label{def:epsilon_discrette}
    \varepsilon = E_V(f) - \frac{C^2}{\sum_{\sigma \in \mathbb{S}_n(\mathcal{M})}\mu(\sigma)},
\end{align}
and the discrete $\varepsilon$-mass-preserving map.
\begin{definition}
    A map $f:\mathcal{M} \to \mathbb{B}^n$ is a discrete $\varepsilon$-mass-preserving map if $\varepsilon = E_V(f) - \frac{C^2}{\sum_{\sigma \in \mathbb{S}_n(\mathcal{M})}\mu(\sigma)}$.
\end{definition}
For the discrete $\varepsilon$-mass-preserving map, we can also give the measurement of $\varepsilon$.

\begin{theorem} \label{thm:epsilon_discrete}
    Let $f:\mathcal{M} \to \mathbb{B}$ be an orientation preserving and piecewise affine map with $\sum_{\sigma\in \mathbb{S}_n(\mathcal{M})} |f(\sigma)| = C$, $f(\partial \mathcal{M}) = \mathbb{S}^{n-1}$ and
    $\varepsilon = E_V(f) - \frac{C^2}{\sum_{\sigma \in \mathbb{S}_n(\mathcal{M})} \mu(\sigma)}$ and
    \begin{align} \label{def:delta}
        \delta = (\delta_1,\delta_2,\cdots), \quad \delta_i = \frac{\nicefrac{|f(\sigma_i)|}{\sum_{\sigma \in \mathbb{S}_n(\mathcal{M})} |f(\sigma_i)|}}{\nicefrac{\mu(\sigma_i)}{\sum_{\sigma \in \mathbb{S}_n(\mathcal{M})} \mu(\sigma)}} - 1,
    \end{align}
    Then
    \begin{align}
        \mu_\text{min} \frac{C^2}{\left( \sum_{\sigma \in \mathbb{S}_n(\mathcal{M})} \mu(\sigma) \right)^2} \|\delta\|_2^2 \leq \varepsilon \leq \mu_\text{max} \frac{C^2}{\left( \sum_{\sigma \in \mathbb{S}_n(\mathcal{M})} \mu(\sigma) \right)^2} \|\delta\|_2^2. \label{ineq:epsilon_discrete}
    \end{align}
    where $\mu_\text{min}$ and $\mu_\text{max}$ are the minimum and maximum of $\mu(\sigma)$ among all $\sigma \in \mathbb{S}_n(\mathcal{M})$, respectively.
\end{theorem}

\begin{proof}
Suppose $m:= \#(\mathbb{S}_n(\mathcal{M}))$. Let
\begin{align}
&\tilde{\xi}_i = \frac{|f(\sigma_i)|}{\sum_{i = 1}^m |f(\sigma_i)|},\quad
\tilde{\mu}_i = \frac{\mu(\sigma_i)}{\sum_{i = 1}^m\mu(\sigma_i)},\\
& \tilde{\mu}_{\text{max}} := \max_i \tilde{\mu}_{i} = \frac{\mu_{\text{max}}}{\sum_{i = 1}^m \mu(\sigma_i)}, \label{eq:tildemumax_discrete}\\
& \tilde{\mu}_{\text{min}} := \max_i \tilde{\mu}_{i} = \frac{\mu_{\text{min}}}{\sum_{i = 1}^m \mu(\sigma_i)}. \label{eq:tildemumin_discrete}
\end{align}
Then $\delta_i = {\tilde{\xi}_i}/{\tilde{\mu}_i} -1$ and
\begin{align}
    \sum_{i = 1}^m \tilde{\xi}_i = \sum_{i = 1}^m \tilde{\mu}_i = 1.
\end{align}
We have
\begin{align}
    \varepsilon = & E_V(f) - \frac{C^2}{\sum_{i = 1}^m \mu(\sigma_i)}
    = \frac{C^2}{\sum_{i = 1}^m\mu(\sigma_i)}  \left[ \sum_{i = 1}^m
    \tilde{\mu}_i(\delta_i + 1)^2 - 1 \right]\\
    = & \frac{C^2}{\sum_{i = 1}^m\mu(\sigma_i)} \sum_{i = 1}^m
    \tilde{\mu}_i[(\delta_i + 1)^2 - 1]
    = \frac{C^2}{\sum_{i = 1}^m\mu(\sigma_i)} \sum_{i = 1}^m
    \tilde{\mu}_i \delta_i^2.  \label{eq:epsilon_proof1_discrete}
\end{align}
Since $\tilde{\mu}_i > 0$, we have
\begin{subequations}
\begin{equation} \label{eq:maxepsilon_discrete}
\sum_{i = 1}^m
\tilde{\mu}_i\delta_i^2 \leq \tilde{\mu}_{\text{max}} \sum_{i = 1}^m \delta_i^2 = \tilde{\mu}_{\text{max}} \|\delta\|_2^2,
\end{equation}
\begin{equation} \label{eq:minepsilon_discrete}
\sum_{i = 1}^m \tilde{\mu}_i\delta_i^2 \geq \tilde{\mu}_{\text{min}} \sum_{i = 1}^m \delta_i^2 = \tilde{\mu}_{\text{min}} \|\delta\|_2^2.
\end{equation}
\end{subequations}
Plugging \eqref{eq:epsilon_proof1_discrete}, \eqref{eq:tildemumax_discrete} and \eqref{eq:tildemumin_discrete} into \eqref{eq:maxepsilon_discrete} and \eqref{eq:minepsilon_discrete}, the inequality \eqref{ineq:epsilon_discrete} is obtained.
\end{proof}

As like the continuous case, \Cref{thm:epsilon_discrete} shows that we can use $\varepsilon$ to measure the mass-preservation of the map. Meanwhile, minimizing $S_f$ in \eqref{set:Sf} is a feasible method to find the $\varepsilon$-mass-preserving map.

\section{Volumetric Stretch Energy Minimization}

We aim to achieve the $n$-dimensional volumetric stretch energy minimization ($n$-VSEM), specifically, minimize $S_f$ in \eqref{set:Sf} to find the $\varepsilon$-mass-preserving map from the $n$-manifold $\mathcal{M}$ to the unit $n$-ball $\mathbb{B}^n$. We rewrite this problem as following.
\begin{align} \label{opt:ball_total}
% \left\{
\begin{array}{l@{}c@{\quad}l}
    & \min_\mathbf{f} & E_V(f):=\frac{1}{n} \trace(\mathbf{f}^\top L_V(f) \mathbf{f})\\
    & \text{s.t.} & \sum_{\sigma\in \mathbb{S}_n(\mathcal{M})} |f(\sigma)| = C\\
    && f(\partial \mathcal{M}) = \mathcal{S}^{n-1}
\end{array}
% \right.
\end{align}
This is a large scale and nonlinear optimization problem. We divide this problem \eqref{opt:ball_total} into boundary and interior parts instead.
Let
\begin{align}
    \mathtt{B} = \{ i~|~ v_i \in \mathbb{S}_0(\partial \mathcal{M})\}, \quad \mathtt{I} = \{ i~|~ v_i \in \mathbb{S}_0(\mathcal{M}) \setminus \mathbb{S}_0(\partial \mathcal{M})\},
\end{align}
and partition $\mathbf{f}$ according to the boundary index $\mathtt{B}$ and interior index $\mathtt{I}$ as $\mathbf{f} = [\mathbf{f}_\mathtt{B}^\top, \mathbf{f}_\mathtt{I}^\top]^\top$.
The boundary and interior subproblems are as following.
\begin{itemize}
    \item[1)] Boundary subproblem: given a measure $\mu'$ defined in $\partial \mathcal{M}$, find the $\varepsilon$-mass-preserving map $g:\partial \mathcal{M} \to \mathcal{S}^{n-1}$ by solving
\begin{align} \label{opt:sphere}
% \left\{
\begin{array}{l@{}c@{\quad}l}
    & \min_\mathbf{g} & E_V(g):=\frac{1}{n-1} \trace(\mathbf{g}^\top L_V(g) \mathbf{g})\\
    & \text{s.t.} & \sum_{\sigma\in \mathbb{S}_{n-1}(\partial \mathcal{M})} |g(\sigma)| = C'\\
    && \|\mathbf{g}_i\|_2^2 = 1, v_i\in \mathbb{S}_0(\partial \mathcal{M})
\end{array}
% \right.
\end{align}
    \item[2)] Interior subproblem: fix the boundary map $\mathbf{f}_\mathtt{B} = \mathbf{g}$ and find the $\varepsilon$-mass-preserving map $f:\mathcal{M} \to \mathbb{B}^n$ by solving
\begin{align} \label{opt:ball}
    \min_{\mathbf{f}_\mathtt{I}} \quad E_V(f):=\frac{1}{n} \trace(\mathbf{f}^\top L_V(f) \mathbf{f}).
\end{align}
\end{itemize}

In the boundary subproblem \eqref{opt:sphere}, the constraint $f(\partial \mathcal{M}) = \mathcal{S}^{n-1}$ is replaced by a weaken constraint $\|\mathbf{g}_i\|^2 = 1$, $v_i \in \partial \mathcal{M}$. This subproblem is a relatively low scale problem compared with \eqref{opt:ball_total} since the number of the boundary vertices is much less than that of whole vertices generally. The boundary subproblem \eqref{opt:sphere} can be treated as an independent problem for the spherical $\varepsilon$-mass-preserving parameterizations from a genus-$0$ $n$-manifold to the unit $n$-hypersphere. This problem is to obtain a boundary map for the computation of the interior map. In the interior subproblem \eqref{opt:ball}, the boundary vertices are fixed. By the divergence theorem, the volume of the $n$-ball is determined by the enclosed surface. Hence, the volume of $f(\mathcal{M})$ is also fixed, that is, the volume sum constraint $\sum_{\sigma\in \mathbb{S}_n(\mathcal{M})} |f(\sigma)| = C$ in \eqref{opt:ball_total} is directed satisfied. Hence, we face to an unconstrained optimization problem in the subproblem \eqref{opt:ball}. As a consequence, we
divide the large scale and nonlinear problem \eqref{opt:ball_total} into a low scale constrained subproblem and an high scale unconstrained subproblem. By the \Cref{thm:Sf_discrete}, the resulting map of the subproblems is the ideal $\varepsilon$-mass-preserving map. The mass-preserving error $\varepsilon$ depends on the boundary subproblem \eqref{opt:sphere}.

\begin{remark}
    Generally, the orientation preserving condition is satisfied for the solution to the subproblems \eqref{opt:sphere} and \eqref{opt:ball} in numerical experiments. Even if there are overlap $n$-simplices, the number of them is little. In this case, a convex combination postprocess can eliminate the overlap $n$-simplices with a tiny mass-preserving loss. Therefore, we ignore the orientation preserving condition in the computation.
\end{remark}

\subsection{Initial map for spherical parameterization}

We first introduce the $n$-dimensional Dirac map inspired by \cite{SHSA00}. Then the initial map for the boundary subproblem \eqref{opt:sphere} is obtained by a north-south alternating iteration. The Dirac map in \cite{SHSA00} is to compute spherical conformal parameterization of the genus-$0$ closed surface by solving an inhomogeneous Laplace-Beltrami equation. We consider the generalized $n-1$ dimensional equation
\begin{align} \label{eq:LBdelta}
    -\Delta_{\partial \mathcal{M}} g = \left(\frac{\partial}{\partial u_1},\frac{\partial}{\partial u_2},\cdots,\frac{\partial}{\partial u_{n-1}}\right) \delta_p
\end{align}
where $p$ is a point on $\partial \mathcal{M}$, $\delta_p$ is Dirac delta function on $p$, $(u_1,u_2,\cdots,u_{n-1})$ is a local orthogonal coordinate on the neighborhood of $p$ and $\frac{\partial}{\partial u_i}$ is the directional differential with respect to $u_i$. The weak formulation of \eqref{eq:LBdelta} for the test function $h$ is
    \begin{align} \label{eq:LBweak}
    \int_{\partial \mathcal{M}}\big\langle\nabla_{\partial \mathcal{M}} g,\nabla_{\partial \mathcal{M}} h\big\rangle ds
    = \Big\langle \left(\frac{\partial}{\partial u_1},\frac{\partial}{\partial u_2},\cdots,\frac{\partial}{\partial u_{n-1}}\right),h \Big\rangle\Big|_p,
\end{align}

Here $\partial \mathcal{M}$ is a simplical $(n-1)$-complex. $g$ and $h$ are piecewise affine map induced by $\mathbf{g}$ and $\mathbf{h}$, respectively. (i) For the left hand term, similar to the proof of \Cref{lma:Dirichlet}, it can be easily verified that
\begin{align} \label{eq:LBweakleft}
    \int_{\partial \mathcal{M}}\big\langle\nabla_{\partial \mathcal{M}} g,\nabla_{\partial \mathcal{M}} h\big\rangle ds = \trace(\mathbf{h}^\top L_D \mathbf{g}),
\end{align}
where $L_D$ is defined in \eqref{eq:LD}.

(ii) For the right hand term, suppose that $p$ is in $\tau_p = [v_0,v_1,\cdots,v_{n-1}]$. Then we need to find an orthogonal coordinate in $\tau_p$, which can be constructed by the QR decomposition of matrix
\begin{align}
[v_{10}^\top,v_{20}^\top,\cdots,v_{(n-1)0}^\top] = QR, \label{eq:qrdecomp}
\end{align}
where $Q = [q_1,q_2,\cdots,q_{n-1}] \in \mathbb{R}^{n\times (n-1)}$ is the required orthogonal coordinates. Hence, the right term becomes
\begin{align} \label{eq:LBweakright}
    \Big\langle \left(\frac{\partial}{\partial u_1},\frac{\partial}{\partial u_2},\cdots,\frac{\partial}{\partial u_{n-1}}\right),h \Big\rangle\Big|_p = \sum_{i = 1}^{n-1} \left.\frac{\partial h^{(i)}}{\partial q_i}\right|_p =
    \trace(\nabla h|_p Q).
\end{align}
Combining \eqref{eq:LBweakleft} and \eqref{eq:LBweakright}, the weak formulation \eqref{eq:LBweak} becomes
\begin{align}
    \trace(\mathbf{h}^\top L_D\mathbf{g}) = \trace(\nabla h|_p Q),
\end{align}
for the test function $h$. Letting
\begin{align}
    &h(v) = \begin{cases}
        e_j^\top, & v = v_i,\\
        0, & v \neq v_i, v\in \mathbb{S}_0(\partial \mathcal{M}),
    \end{cases}
\end{align}
for $i = 1,\cdots,N$ and $j = 1,2,\cdots,n$, we have
\begin{align}
    L_D\mathbf{g} = \mathbf{b}, \label{eq:dirac_discrete}
\end{align}
where
\begin{align} \label{eq:b}
    [\mathbf{b}]_i = \begin{cases}
        0,& v_i \notin \mathbb{S}_0(\tau_p), \\
        \nabla \alpha_i Q, & v_i \in \mathbb{S}_0(\tau_p).
    \end{cases}
\end{align}
By the barycentre coodinates formula in \eqref{eq:NablaLambda}, letting $R = [r_1,r_2,\cdots,r_{n-1}]$ in \eqref{eq:qrdecomp} and $r_0 = 0$, we have
\begin{align} \label{eq:b_detail}
\nabla \alpha_i Q = \frac{1}{|\tau_p|}\sum_{j \neq i} \tilde{w}_{ij}^{\tau_p} v_{ij}Q =
\frac{1}{|\tau_p|} \sum_{j \neq i} \tilde{w}_{ij}^{\tau_p} (r_i - r_j)^\top.
\end{align}

To solve the linear system \eqref{eq:dirac_discrete} under the spherical constraint, we introduce the $(n-1)$-dimensional stereographic projection $\Pi:\mathcal{S}^{n-1} \to \overline{\mathbb{R}}^{n-1}$, defined as
\begin{align}
    \Pi(\mathbf{f}_i) = \left[ \frac{\mathbf{f}_i^{1}}{1-\mathbf{f}_i^{n}}, \frac{\mathbf{f}_i^{2}}{1-\mathbf{f}_i^{n}}, \cdots, \frac{\mathbf{f}_i^{n-1}}{1-\mathbf{f}_i^{n}} \right],
\end{align}
for $\mathbf{f}_i = [\mathbf{f}_i^1, \mathbf{f}_i^2, \cdots, \mathbf{f}_i^n]$. Let $h = \Pi \circ g$ and $\mathbf{h} = \Pi(\mathbf{g})$. Consider the linear system
\begin{align}
    L_D \mathbf{h} = \mathbf{b}.
\end{align}
The Laplacian matrix $L_D$ is singular, which has a eigenpair $(0,\mathbf{1})$, which means that it is unique up to translation. Hence, we can fix a point $\mathbf{h}_i$ to solve it. Since the stereographic projection is a conformal map, the map $g := \Pi^{-1}\circ h$ is the solution of \eqref{eq:dirac_discrete}. \Cref{alg:Dirac} shows the computation of the $(n-1)$-dimensional spherical Dirac map.

\begin{algorithm}[h]
\caption{Spherical parameterization by Dirac map}
\begin{algorithmic}[1] \label{alg:Dirac}
    \REQUIRE $(n-1)$-simplicial complex $\partial \mathcal{M}$ topologically equivalent to $\mathcal{S}^{n-1}$.
    \ENSURE A spherical Dirac parameterization $g: \partial \mathcal{M}\to \mathcal{S}^{n-1}$.
    \STATE Select the most regular $(n-1)$-simplex as $\tau_p$ containing the Dirac point.
    \STATE Construct the right hand term $\mathbf{b}$ as \eqref{eq:b} and \eqref{eq:b_detail}.
    \STATE Let $N$ be the number of vertices. Fix a vertex $\mathbf{h}_i = \mathbf{0}$ and set $\hat{\mathtt{I}} = \{1,2,\cdots,N\}\setminus \{i\}$.
    \STATE Compute $\mathbf{h}$ by solving the linear system
    \begin{align}
        [L_D]_{\hat{\mathtt{I}}\hat{\mathtt{I}}} \mathbf{h}_{\hat{\mathtt{I}}} = \mathbf{b}_{\hat{\mathtt{I}}},
    \end{align}
    where $L_D$ is defined as \eqref{eq:LD}.
    \STATE Perform the centralization $\mathbf{h} \gets \mathbf{h} - \frac{1}{N}\mathbf{1}_N \mathbf{1}_N^\top \mathbf{h}$.
    \STATE Obtain the spherical Dirac map by the inverse stereographic projection $\mathbf{g} = \Pi^{-1}(\mathbf{h})$.
\end{algorithmic}
\end{algorithm}

Inspired from SEM algorithm in \cite{MHTL19}, we can apply the stereographic projection to design a north-south alternating iteration algorithm in \ref{alg:SEM}. The resulting map of this algorithm has small $\varepsilon$. Hence, it can provide an $\varepsilon$-mass-preserving map as the initial map and the sum of the volume $C'$ for our proposed algorithm.

\begin{algorithm}[H]
\caption{SEM for spherical mass-preserving parameterization}
\begin{algorithmic}[1] \label{alg:SEM}
    \REQUIRE $(n-1)$-simplicial complex $\partial \mathcal{M}$ topologically equivalent to $\mathcal{S}^{n-1}$, interior radius $r$, and a tolerance $tol$.
    \ENSURE A $\varepsilon$-mass-preserving parameterization $g:\partial \mathcal{M}\to \mathbb{S}^{n-1}$ induced by $\mathbf{g}$.
    \STATE Let $m = \#\mathbb{S}_0(\partial \mathcal{M})$.
    \STATE Compute a Dirac map $g$ by \Cref{alg:Dirac}.
    \STATE Compute $\mathbf{h}_i = \Pi (\mathbf{g}_i)$, $i = 1,2,\cdots,m$.
    \STATE Let $E_{\text{old}} \gets E_V(g)$ and $\delta E = +\infty$.
    \WHILE{$\delta E > tol$}
    \STATE Update $L \gets L_V(g)$ as in \eqref{eq:LVf}.
    \STATE $\mathbf{h} \gets \diag (|\mathbf{h}|^{-2}) \mathbf{h}$.
    \STATE Set $\mathtt{I} = \{i~|~ |\mathbf{h}_i| < r\}$, $\mathtt{B} = \{1,2,\cdots,m\}\setminus \mathtt{I}$.
    \STATE Update $\mathbf{h}$ by solving $[ L_S ]_{\mathtt{I}\mathtt{I}} \mathbf{h}_{\mathtt{I}} = -[ L_S ]_{\mathtt{I}\mathtt{B}} \mathbf{h}_{\mathtt{B}}$.
    \STATE Compute $\mathbf{g}_i = \Pi^{-1}(\mathbf{h}_i)$, $i = 1,2,\cdots, m$.
    \STATE Let $E_{\text{new}} \gets E_V(g)$ and $\delta E = E_{\text{old}} - E_{\text{new}}$.
    \ENDWHILE
\end{algorithmic}
\end{algorithm}

\subsection{Spherical $\varepsilon$-mass-preserving parameterization}

The spherical $\varepsilon$-mass-preserving parameterization is achieved by solving the spherical subproblem \eqref{opt:sphere} as following,
\begin{align} %\label{opt:sphere}
% \left\{
\begin{array}{l@{}c@{\quad}l}
    & \min_\mathbf{g} & E_V(g):=\frac{1}{n-1} \trace(\mathbf{g}^\top L_V(g) \mathbf{g})\\
    & \text{s.t.} & \sum_{\sigma\in \mathbb{S}_{n-1}(\partial \mathcal{M})} |g(\sigma)| = C'\\
    && \|\mathbf{g}_i\|_2^2 = 1, v_i\in \mathbb{S}_0(\partial \mathcal{M})
\end{array}
% \right.
\end{align}
We adopt the Newton method to this problem. We first present a theorem for the gradient of the $n$-VSE.
\begin{theorem} \label{thm:gradient}
The gradient of volumetric stretch energy functional in \eqref{eq:EVcot} can be formulated as
\begin{align} \label{eq:GradE}
    \nabla E_V(f) = 2 L_V(f) \mathbf{f}.
\end{align}
\end{theorem}

\begin{proof}
By \Cref{thm:EVf}, we have
\begin{align}
    E_V(f) =  \sum_{\sigma\in\mathbb{S}_n(\mathcal{M})} \frac{|f(\sigma)|^2}{\mu(\sigma)}.
\end{align}
Using the chain rule, we can represent its gradient as
\begin{align}
    \nabla E_V(f) = 2 \sum_{\sigma\in\mathbb{S}_n(\mathcal{M})} \frac{|f(\sigma)|}{\mu(\sigma)} \nabla |f(\sigma)|.
\end{align}
Plugging the gradient formula of $|f(\sigma)|$ in \eqref{eq:nablafsigma} into it, we have
\begin{align}
    \frac{\partial}{\partial \mathbf{f}_i} E_V(f) = &2 \sum_{\sigma\in\mathbb{S}_n(\mathcal{M})}  \sum_{[v_i,v_j] \in \mathbb{S}_1(\sigma)} \left[\frac{1}{n(n-1)}\frac{|f(\sigma_{-ij})|}{\sigma_{\mu,f^{-1}}(\sigma)} \cot\theta_{ij}^{\sigma}(f) \right] \mathbf{f}_{ij}\\
     = & 2 \sum_{\sigma\in\mathbb{S}_n(\mathcal{M})} \sum_{[v_i,v_j] \in \mathbb{S}_1(\sigma)} w_{ij}^{\sigma}(f) (\mathbf{f}_i - \mathbf{f}_j)
     = 2 \sum_{[v_{i},v_j] \in \mathbb{S}_1(\mathcal{M})} w_{ij}(f) (\mathbf{f}_i - \mathbf{f}_j)\\
     =& 2 e_i^\top L_V(f) \mathbf{f},
\end{align}
completing the proof.
\end{proof}
The gradient formula in \eqref{eq:GradE} also holds for $\mathcal{S}^{n-1}$ without loss of generality. Then, we consider the Lagrange multiplier,
\begin{align} \label{eq:lagrange}
    L(\mathbf{g},\lambda,\boldsymbol{s}) = E_V(g) + \lambda (\sum_{\sigma\in \mathbb{S}_{n-1}(\partial \mathcal{M})} |g(\sigma)| - C') + \frac{1}{2}\sum_i \boldsymbol{s}_i (\|\mathbf{g}_i\|_2^2 - 1).
\end{align}
From \Cref{lma:volume} and \Cref{thm:gradient}, the Karush–Kuhn–Tucker equations are
\begin{align}
    &2L_V(g)\mathbf{g} + \lambda L_D(g) \mathbf{g} + \diag(\boldsymbol{s}) \mathbf{g} = 0,\\
    &\sum_{\sigma\in \mathbb{S}_{n-1}(\partial \mathcal{M})} |g(\sigma)| - C' = 0,\\
    & \frac{1}{2}(\|\mathbf{g}\|^2 - 1) = 0,
\end{align}
where $\|\mathbf{g}\|$ means the column vector with the $i$-th entry being $\|\mathbf{g}_i\|_2$. For this primal-dual problem, the Newton step is obtained by solving the linear system
\begin{align} \label{eq:newton}
\begin{bmatrix}
    \nabla_{\text{vec}(\mathbf{g})}^2 L & \text{vec}(L_D(g) \mathbf{g}) & \text{cdiag}(\mathbf{g}) \\
    \text{vec}(L_D(g) \mathbf{g})^\top & 0 & 0 \\
    \text{cdiag}(\mathbf{g})^\top & 0 & 0
\end{bmatrix}
\begin{bmatrix}
    \Delta \text{vec}(\mathbf{g}) \\ \Delta \lambda \\ \Delta \boldsymbol{s}
\end{bmatrix}
= -
\begin{bmatrix}
    \text{vec}(2L_V(g)\mathbf{g} + \lambda L_D(g) \mathbf{g} + \mathbf{g} \diag(\boldsymbol{s})) \\
    \sum_{\sigma\in \mathbb{S}_{n-1}(\partial \mathcal{M})} |g(\sigma)| - C'\\
    \|\mathbf{g}\|^2 - 1
\end{bmatrix},
\end{align}
in which $\text{vec}$ means the vectorization of a matrix, $\text{cdiag}(\mathbf{g})$ is a block diagonal matrix with its $i$-th block being $\mathbf{g}^{i}$, and $\nabla_{\text{vec}(\mathbf{g})}^2 L$ is the Hessian matrix with respect to ${\text{vec}(\mathbf{g})}$ of the Lagrange multiplier \eqref{eq:lagrange}.
It can be verified that the Hessian matrix $\nabla_\mathbf{g}^2 L$ is of identical sparsity as $\mathbf{1}_{(n-1)\times(n-1)}
\otimes L_D$.
Here we can estimate it by finite difference method.
Hence, we propose the \Cref{alg:SMP} for the $(n-1)$-spherical mass-preserving parameterization.
\begin{algorithm}[h]
\caption{$(n-1)$-VSEM for the $(n-1)$-spherical mass-preserving parameterization}
\begin{algorithmic}[1] \label{alg:SMP}
    \REQUIRE $(n-1)$-simplicial complex $\partial \mathcal{M}$ topologically equivalent to $\mathcal{S}^{n-1}$ with the measure $\mu$, tolerance $tol$.
    \ENSURE A spherical mass-preserving parameterization $g: \partial \mathcal{M}\to \mathcal{S}^{n-1}$.
    \STATE Compute a Dirac map $g$ by \Cref{alg:SEM} and set $C' = |g(\partial \mathcal{M})|$.
    \STATE Let $E_{\text{old}} = E_V(g)$ and $\delta E = +\infty$.
    \WHILE{$\delta E< tol$}
    \STATE Compute the Newton step by solving the linear system \eqref{eq:newton}.
    \STATE Find a step size $\alpha$ by the linear search and update
    \begin{align}
        \mathbf{g} \gets \mathbf{g} + \alpha \Delta \mathbf{g},\quad
        \lambda \gets \lambda + \alpha \Delta \lambda,\quad
        \boldsymbol{s} \gets \boldsymbol{s} + \alpha \Delta \boldsymbol{s}.
    \end{align}
    \STATE Let $E_{\text{new}} \gets E_V(g)$ with $g$ induced by $\mathbf{g}$ and $\delta E = E_{\text{old}} - E_{\text{new}}$. Set $E_{\text{old}} \gets E_{\text{new}}$.
    \ENDWHILE
\end{algorithmic}
\end{algorithm}

\subsection{Ball $\varepsilon$-mass-preserving parameterization}

The ball $\varepsilon$-mass-preserving parameterization need a fixed boundary map according to the above discussion.
However, a sliver $n$-manifold may lead to a boundary map with low quality in the numerical perspective. To obtain a better boundary map, we stretch the boundary vertices $V_\mathtt{B} := [v_1^\top, v_2^\top, \cdots, v_m^\top]^\top$ along the principal axes, such that the $n$-manifold becomes a ball-like shape. Generally, this stretched $n$-manifold may result in a better boundary map. Moreover, this stretch transformation is also a mass-preserving map.
Specifically, the stretch transformation
is obtained by solving the eigenvalue problem
\begin{align}
    \tilde{V}_\mathtt{B}^\top \tilde{V}_\mathtt{B}X = X\Lambda,
\end{align}
where $\tilde{V}_\mathtt{B} = (I - \frac{1}{m}\mathbf{1}\mathbf{1}^\top)V_\mathtt{B}$ with $m = \# \mathtt{B}$, $X = [x_1,\cdots,x_n]$ is the directions of the principal axes and $\Lambda = [\lambda_1,\cdots,\lambda_n]$ are the square of length of the principal axes. Then $\tilde{V}_\mathtt{B}X \Lambda^{-\frac{1}{2}}$ are the vertices of the ball-like manifold. The solution of this problem can be calculated by the singular value decomposition of $\tilde{V}_\mathtt{B}$. We can verify that the left singular vector $U = \tilde{V}_\mathtt{B}X \Lambda^{-\frac{1}{2}}$.

After the computation of the boundary map $f_\mathtt{B} = \mathbf{g}$, we focus on the ball subproblem \eqref{opt:ball} as following,
\begin{align} \label{opt:ball2}
    \min_{\mathbf{f}_\mathtt{I}} \quad E_V(f):=\frac{1}{n} \trace(\mathbf{f}^\top L_V(f) \mathbf{f}).
\end{align}
Different from the spherical subproblem, since the boundary vertices are fixed, the volume of the enclosed region of the boundary surface is determined. Hence, the conditions in \eqref{set:Sf} are satisfied. By the \Cref{thm:Sf_discrete}, the solution to the unconstrained problem \eqref{opt:ball2} is the target $\varepsilon$-mass-preserving map.

Partition $L_V(f)$ according to the boundary index $\mathtt{B}$ and interior index $\mathtt{I}$
\begin{align}
    L_V(f) = \begin{bmatrix}
        [L_V(f)]_{\mathtt{B}\mathtt{B}} & [L_V(f)]_{\mathtt{B}\mathtt{I}}\\
        [L_V(f)]_{\mathtt{I}\mathtt{B}} &
        [L_V(f)]_{\mathtt{I}\mathtt{I}}
    \end{bmatrix}.
\end{align}
From the \Cref{thm:gradient}, the subproblem \eqref{opt:ball2} reaches the minimum if $[L_V(f)\mathbf{f}]_\mathtt{I} = 0$, that is,
\begin{align} \label{eq:le_ball}
    [L_V(f)]_{\mathtt{I}\mathtt{I}} \mathbf{f}_{\mathtt{I}} = -[L_V(f)]_{\mathtt{I}\mathtt{B}} \mathbf{f}_{\mathtt{B}}.
\end{align}
This motivate us to solve the linear system \eqref{eq:le_ball} to compute the interior vertices $\mathbf{f}_\mathtt{I}^{(i)}$ by the Laplacian matrix $L_V(f^{(i-1)})$ at the $i$-th iterative step. We summarize this strategy in \Cref{alg:BMP}.

\begin{algorithm}[h]
\caption{$n$-VSEM for the $n$-ball mass-preserving parameterization}
\begin{algorithmic}[1] \label{alg:BMP}
    \REQUIRE $n$-simplicial complex $\mathcal{M}$ topologically equivalent to $\mathbb{B}^{n}$ with the measure $\mu$, tolerance $tol$.
    \ENSURE An $n$-ball mass-preserving parameterization $f: \mathcal{M}\to \mathbb{B}^{n}$ induced by $\mathbf{f}$.
    \STATE Let $\mathtt{B} = \{ i~|~ v_i \in \mathbb{S}_0(\partial \mathcal{M})\}$, and $\mathtt{I} = \{ i~|~ v_i \in \mathbb{S}_0(\mathcal{M}) \setminus \mathbb{S}_0(\partial \mathcal{M})\}$. Set $m = \# \mathtt{B}$.
    \STATE Compute the left singular matrix $U$ of $\tilde{V}_\mathtt{B} = (I - \frac{1}{m}\mathbf{1}\mathbf{1}^\top) V_\mathtt{B}$ and set $V_\mathtt{B} \gets U$.
    \STATE Compute a spherical mass-preserving map $g:\partial \mathcal{M} \to \mathcal{S}^{n-1}$ with a measure $\mu'$ by \Cref{alg:SMP}.
    \STATE Construct the Laplacian matrix $L \gets L_D$ as \eqref{eq:LD} for the $n$-manifold $\mathcal{M}$.
    \STATE Fix $\mathbf{f}_\mathtt{B} = \mathbf{g}$ and Compute the interior vertices by solving the linear system
    $
        L_{\mathtt{I}\mathtt{I}} \mathbf{f}_{\mathtt{I}} = -L_{\mathtt{I}\mathtt{B}} \mathbf{f}_{\mathtt{B}}.
    $
    \STATE Let $E_\text{old} = E_V(f)$ with $f$ induced by $\mathbf{f}$ and $\delta E = +\infty$.
    \WHILE{$\delta E > tol$}
    \STATE Update $L \gets L_V(f)$ as in \eqref{eq:LVf}.
    \STATE Update $\mathbf{f}_\mathtt{I}$ by solving the linear system
    $
        L_{\mathtt{I}\mathtt{I}} \mathbf{f}_{\mathtt{I}} = -L_{\mathtt{I}\mathtt{B}} \mathbf{f}_{\mathtt{B}}.
    $
    \STATE Let $E_\text{new} \gets E_V(f)$ and $\delta E = E_\text{old} - E_\text{new}$. Set $E_\text{old} \gets E_\text{new}$.
    \ENDWHILE
\end{algorithmic}
\end{algorithm}

\section{Numerical Experiments}

In this section, we indicate the numerical performance of our proposed algorithms for the $n$-manifold $\mathcal{M}$ and the $(n-1)$-manifold. We mainly focus on the $3$-dimensional and $4$-dimensional cases. The $3$-dimensional surface benchmarks are taken from AIM@SHAPE shape repository (\url{http://visionair.ge.imati.cnr.it/ontologies/shapes/}),
the Stanford 3D scanning repository (\url{http://graphics.stanford.edu/data/3Dscanrep}),
and TurboSquid (\url{https://www.turbosquid.com/}). The $3$-dimensional solid benchmarks are generated with the above surfaces by the builtin function \emph{generateMesh} of MATLAB.
The $4$-dimensional benchmarks are manually generated. All experimental programs are executed in MATLAB R2021a on a personal computer with a 2.50 GHz CPU and 64 GB RAM. We consider the volume-preserving map if no special illustration. In other words, the measures of the benchmarks are $\mu(\cdot) = |\cdot|$.

\subsection{Ellipsoid case}
First, we give $3$- and $4$-ellipsoids as a special case for the $n$-VSEM. Let $E^n(\mathbf{a})$ be a $n$-ellipsoid with $\mathbf{a}_i$ being the length of the $i$-th axis. We can find that the map $f^*:(v^1,v^2,\cdots,v^n) \to \left( \frac{v_1}{\mathbf{a}_1}, \frac{v_2}{\mathbf{a}_2}, \cdots, \frac{v_n}{\mathbf{a}_n} \right)$ is a volume-preserving map. We aim to compute the this volume-preserving map $f^*$ to confirm the mass-preservation of $n$-VSEM. It is worth noting that the boundary map $f^*(\partial E^n(\mathbf{a})) = \mathcal{S}^{n-1}$ is not a volume-preserving map. The ratio of the volume $|f(\tau)|/|\tau|$ with $\tau = [v_0,v_1,\cdots,v_{n-1}]$ satisfies
\begin{align}
    \frac{|f^*(\tau)|}{|\tau|} = \frac{|\mathbf{f}_{10}^* \wedge \mathbf{f}_{20}^* \wedge \cdots \wedge \mathbf{f}_{n0}^*|}{|v_{10} \wedge v_{20} \wedge \cdots \wedge v_{n0}|} = \frac{\sqrt{\sum_{i = 1}^{n-1} \det([ \mathbf{f}_{(i+1)0}^{*\top}, \cdots, \mathbf{f}_{(n-1)0}^{*\top}, \mathbf{f}_{10}^{*\top}, \cdots, \mathbf{f}_{(i-1)0}^{*\top} ])}}{\sqrt{\sum_{i = 1}^{n-1} \det([ v_{(i+1)0}^\top, \cdots, v_{(n-1)0}^\top, v_{10}^\top, \cdots, v_{(i-1)0}^\top ])}}.
\end{align}
Hence, we choose $\mu'(\tau) = \frac{|\tau|}{|f^*(\tau)|}$ and $C' = \sum_{\tau \in \mathbb{S}_{n-1}(\partial \mathcal{M})} |f^*(\tau)|$ for the boundary map in the \Cref{alg:SMP}. The initial map is the exact solution with a perturbation $10^{-4}$. Then we compute the map $f$ by the \Cref{alg:BMP} with the computed boundary map. Since the normalization for the ball-like shape in \Cref{alg:BMP} directly leads to the exact solution, we do not execute this normalization process to show the effect of the proposed algorithms.
The \Cref{alg:SMP} converges within $20$ iterative steps and \Cref{alg:BMP} converges  within $3$ iterative steps.
\Cref{tab:ellip} shows the $3$ and $4$ dimensional ellipsoids for the experiment and the mass-preserving performances of \Cref{alg:SMP} and \Cref{alg:BMP}, where $\# \mathtt{B}$ and $\# \mathtt{I}$ are the the number of the boundary and vertices vertices, respectively. And the definitions of $\varepsilon$ and $\delta$ are in \eqref{def:epsilon_discrette} and \eqref{def:delta}, respectively. We can see that the measurements $\varepsilon$ and $\delta$ are closed to $10^{-16}$, which indicates that our proposed algorithms compute the mass-preserving parameterization numerically.

\begin{table}[h]
    \centering
    \begin{tabular}{c|c|cc|ccc|ccc}
    \hline
        \multirow{2}{*}{$n$} & \multirow{2}{*}{$\mathbf{a}$} & \multirow{2}{*}{$\# \mathtt{B}$} & \multirow{2}{*}{$\# \mathtt{I}$} & \multicolumn{3}{c|}{Sphere} & \multicolumn{3}{c}{Ball} \\ \cline{5-10}
         & && & $\varepsilon$ & $\text{mean}(\delta)$ & $\text{SD}(\delta)$ & $\varepsilon$ & $\text{mean}(\delta)$ & $\text{SD}(\delta)$\\
    \hline
        \multirow{2}{*}{$3$} & $(0.8,1,1.2)$ & 6707 & 48811 & 3.5e-15 & 4.2e-11 & 1.6e-8 & 2.4e-14 & 6.6e-11 & 7.8e-8 \\
        & $(0.5,1,1.5)$ & 7210 & 42489 & 2.3e-14 & 1.6e-11 & 9.5e-9 & 2.3e-14 & 3.8e-11 & 6.8e-8 \\
    \hline
        \multirow{2}{*}{$4$} & $(0.7,0.9,1.1,1.3)$ & 4003 & 40030 & 5.4e-15 & 2.5e-12 & 1.7e-10 & 5.7e-14 & 1.9e-12 & 7.69e-10 \\
        & $(0.5,0.8,1.1,1.4)$ & 5121 & 45000 & 4.6e-14 & 7.1e-12 & 1.6e-9 & 1.3e-12 & 1.0e-11 & 1.26e-8\\
    \hline
    \end{tabular}
    \caption{The meshes of $3$ and $4$ dimensional ellipsoids with the axis lengths and the number of boundary and interior vertices. The columns Sphere and Ball are the measurements of the resulting maps of \Cref{alg:SMP} for the boundary map and \Cref{alg:BMP} for the whole map.}
    \label{tab:ellip}
\end{table}

\subsection{General cases}
Now we focus on the general cases, which have not exact volume-preserving solutions. Hence, we compute the $\varepsilon$-volume-preserving parameterization with as small $\varepsilon$ as possible. \Cref{fig:mesh} shows the $3$ dimensional benchmarks and the projected surfaces along a direction of the $4$ dimensional benchmarks. \Cref{tab:generalmesh} demonstrates the basic information and the results of the proposed algorithms on the $3$ and $4$ dimensional benchmarks, where $n$ is the dimension of the manifolds. We set $|\partial \mathcal{M}| = |\mathcal{S}^{n-1}|$ and $|\mathcal{M}| = |\mathbb{B}^{n}|$ in the measurement $\varepsilon$ for the better presentation and comparison.

\begin{figure}[thp]
    \centering
\resizebox{\textwidth}{!}{
\begin{tabular}{c@{}c@{}c@{}c@{}}
        \includegraphics[clip,trim = {15cm 2cm 15cm 2cm},width = 0.24\textwidth]{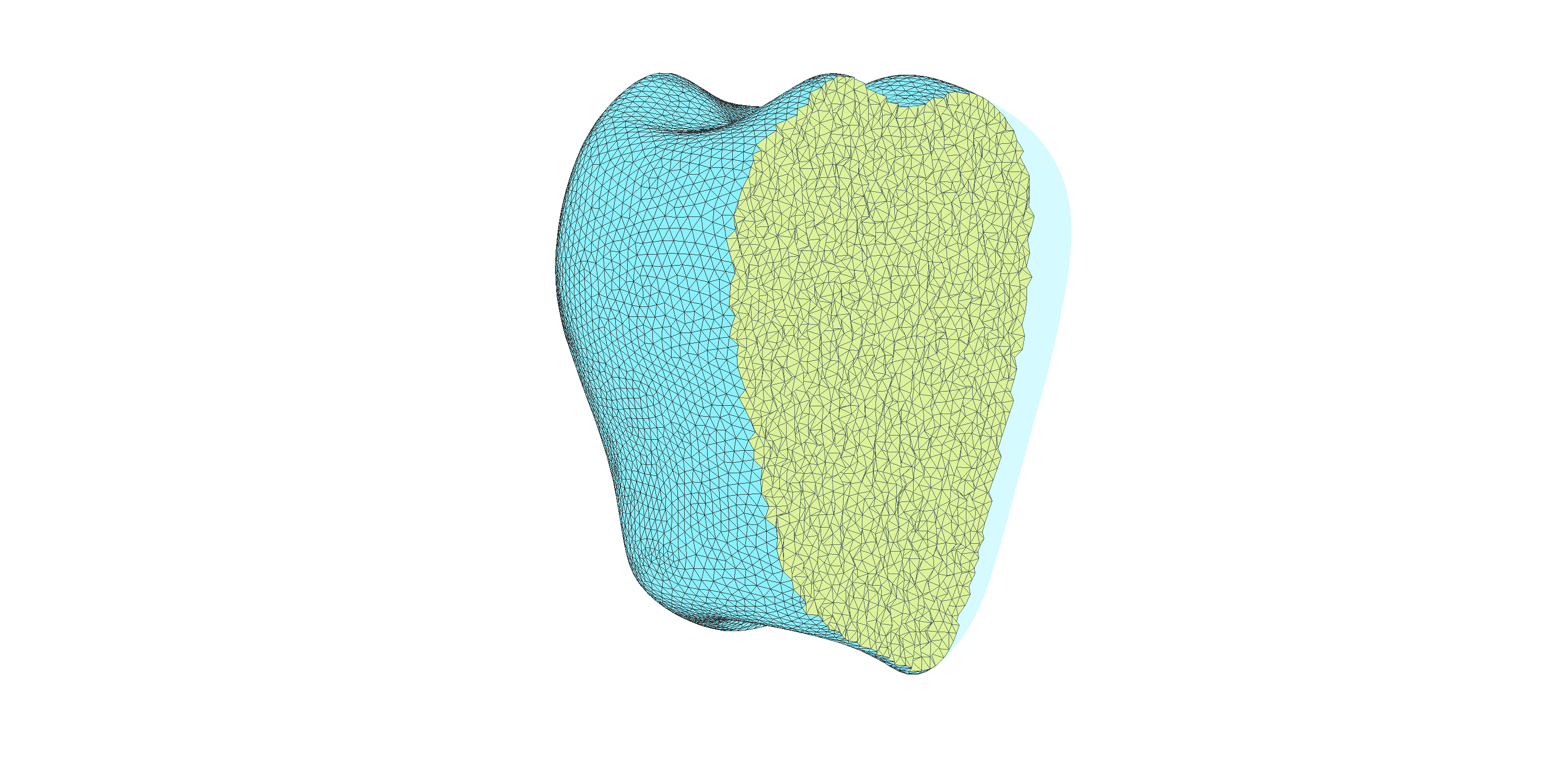} &
        \includegraphics[clip,trim = {15cm 2cm 15cm 2cm},width = 0.24\textwidth]{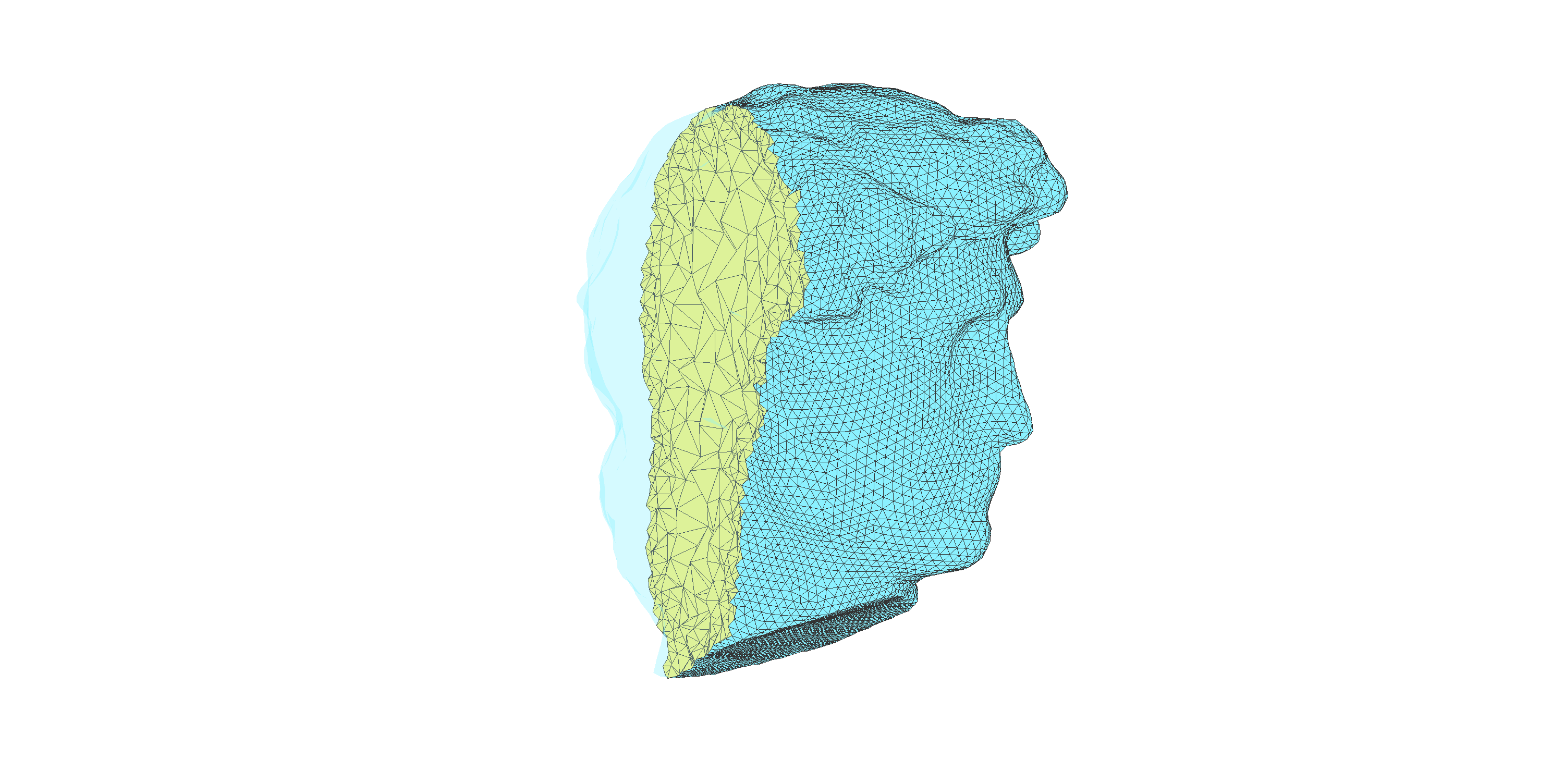} &
        \includegraphics[clip,trim = {17cm 5cm 17cm 3.5cm},width = 0.24\textwidth]{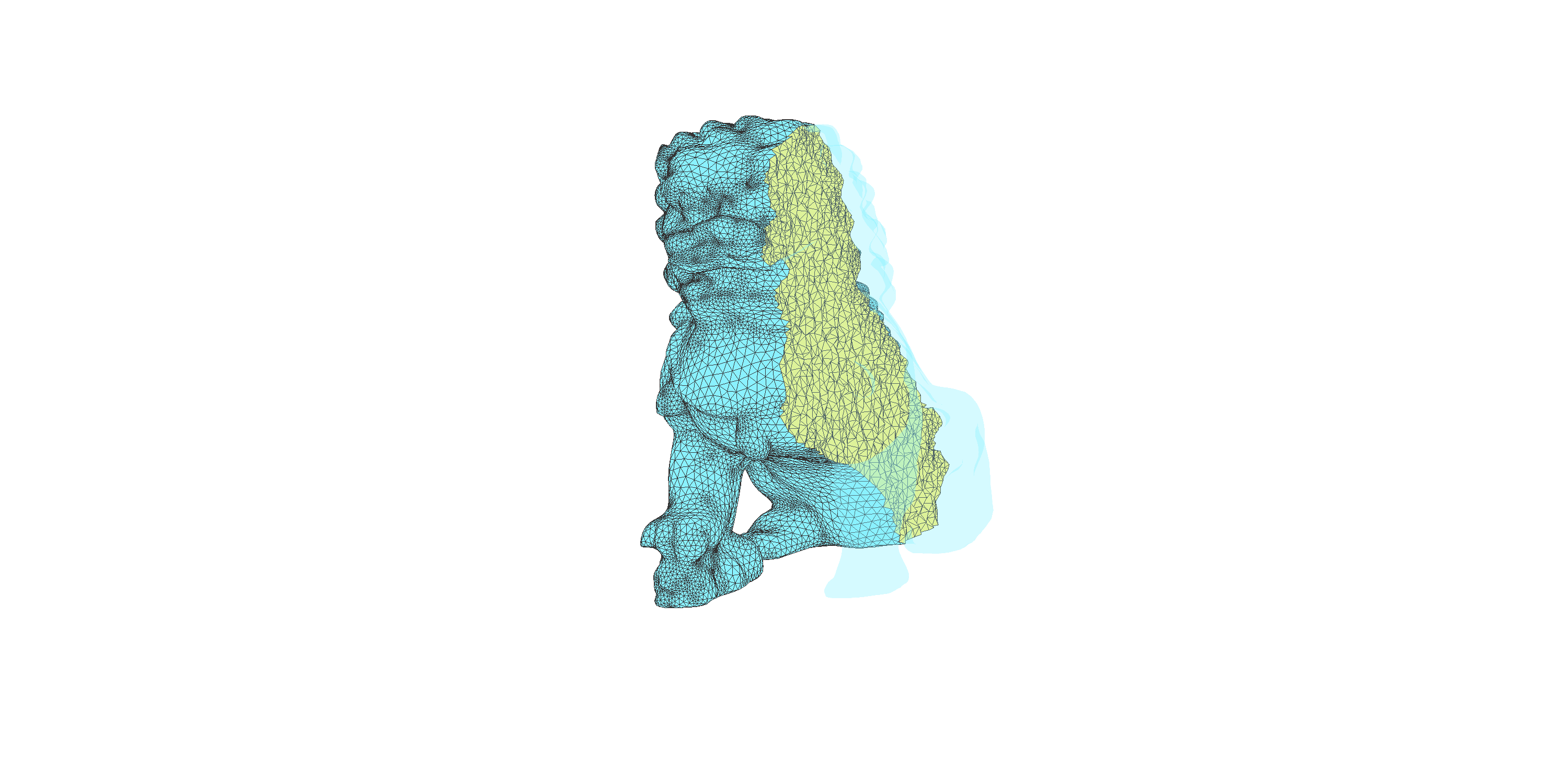} &
        \includegraphics[clip,trim = {15cm 2cm 15cm 2cm},width = 0.24\textwidth]{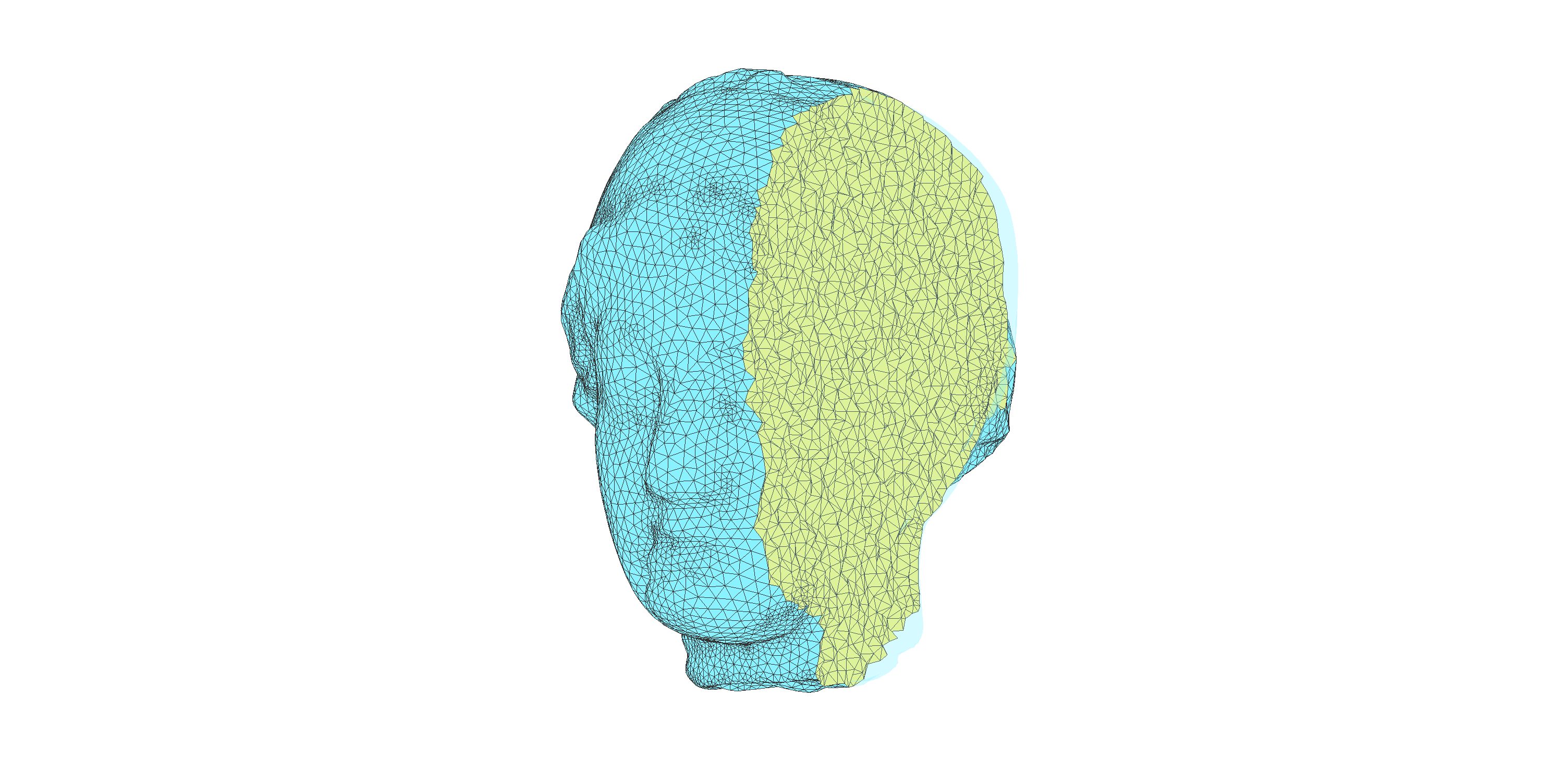} \\
        Apple & DavidHead & LionStatue & VenusHead\\
        &
        \includegraphics[clip,trim = {13cm 2cm 11cm 4cm},width = 0.24\textwidth]{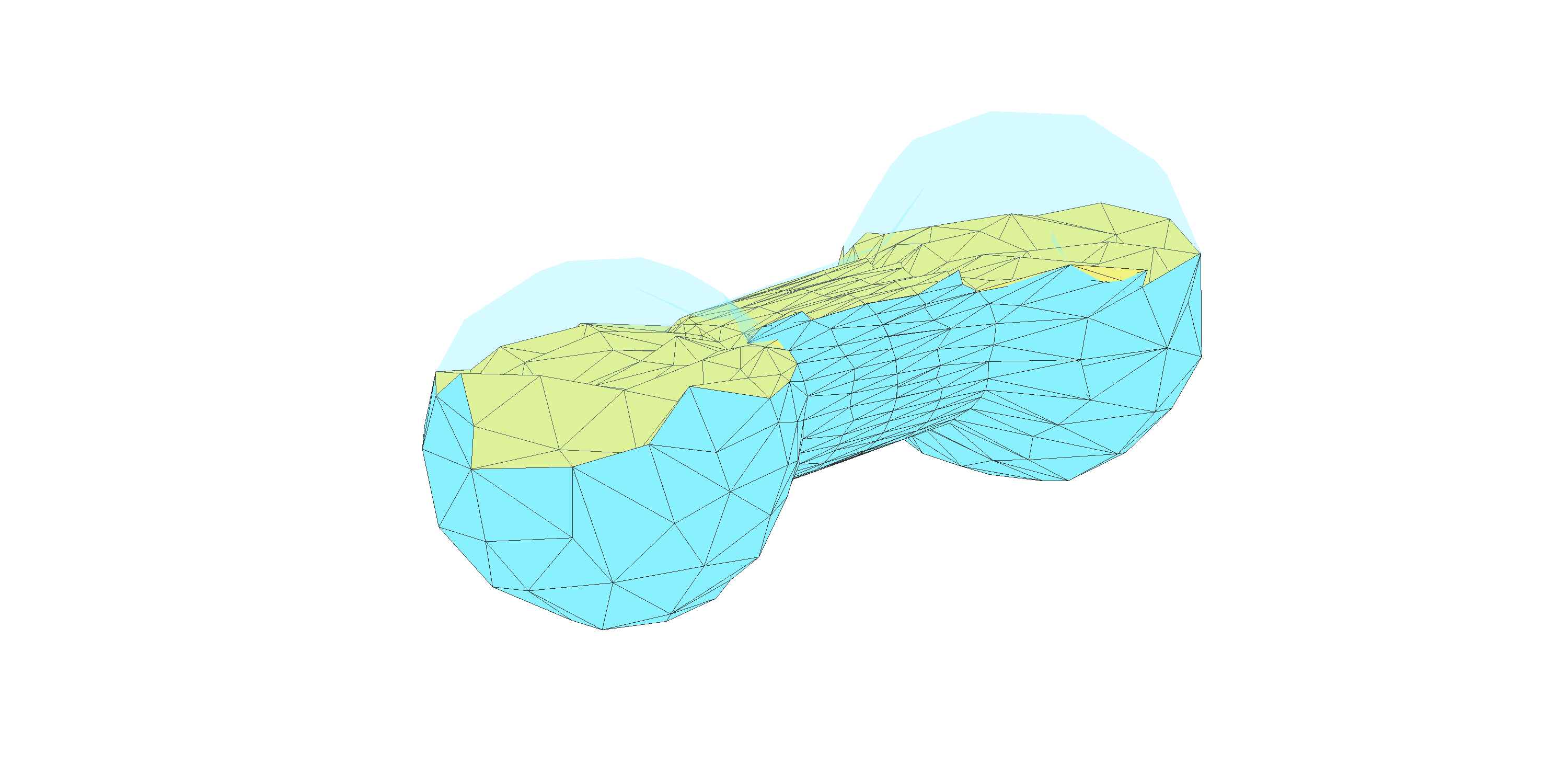} &
        \includegraphics[clip,trim = {15cm 2cm 15cm 2cm},width = 0.24\textwidth]{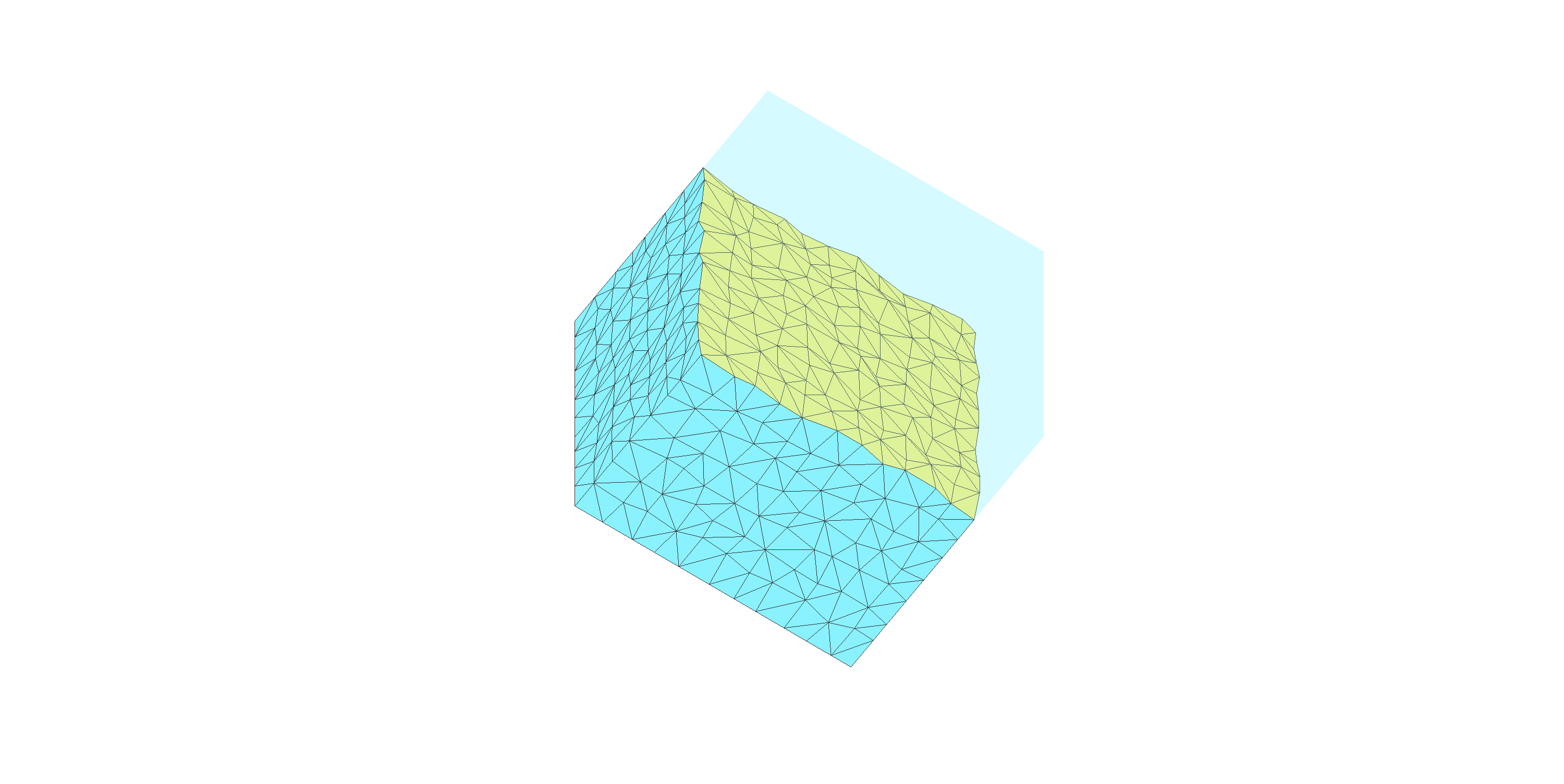} & \\
        & Dumbbell & Tesseract & \\
\end{tabular}
}
\caption{The $3$ and $4$ dimensional benchmarks for the experiments. The first row is $3$ dimensional benchmarks and the second row is the projection along a direction of $4$ dimensional benchmarks.}
    \label{fig:Histogram}
\end{figure}

\begin{table}[h]
    \centering
    \begin{tabular}{c|c|cc|ccc|ccc}
    \hline
        \multirow{2}{*}{$n$} & \multirow{2}{*}{Mesh} & \multirow{2}{*}{$\# \mathtt{B}$} & \multirow{2}{*}{$\# \mathtt{I}$} & \multicolumn{3}{c|}{Sphere} & \multicolumn{3}{c}{Ball} \\ \cline{5-10}
         & && & $\varepsilon$ & $\text{mean}(\delta)$ & $\text{SD}(\delta)$ & $\varepsilon$ & $\text{mean}(\delta)$ & $\text{SD}(\delta)$\\
    \hline
        \multirow{4}{*}{$3$} & Apple & 7947 & 54185 & 1.2e-6 & 1.2e-6 & 3.8e-4 & 1.0e-4 & 4.2e-5 & 1.6e-2 \\
        & DavidHead  & 10671 & 9313 & 2.5e-6 & 3.4e-7 & 4.7e-4 & 2.4e-2 & 8.4e-3 & 1.2e-1 \\
        & LionStatue & 23057 & 50080 & 1.9e-4 & 6.8e-5 & 5.8e-3 & 3.2e-2 & 7.1e-3 & 1.2e-1 \\
        & VenusHead & 9129 & 36105 & 2.8e-6 & 1.9e-6 & 9.5e-4 & 3.9e-3 & 1.1e-3 & 4.4e-2 \\
    \hline
        \multirow{2}{*}{$4$} & Dumbbell & 5043 & 45387 & 1.3e-1 & 3.4e-3 & 1.3e-1 &  &  &  \\
        & Tesseract & 10736 & 10000 & 8.1e-2 & 5.9e-4 & 6.7e-2 & 3.0e-1 & 2.1e-3 & 2.4e-1\\
    \hline
    \end{tabular}
    \caption{The $3$ and $4$ dimensional benchmarks with the number of boundary and interior vertices. The columns Sphere and Ball are the measurements of the resulting maps of \Cref{alg:SMP} for the boundary maps and \Cref{alg:BMP} for the whole maps.}
    \label{tab:generalmesh}
\end{table}

\begin{figure}[thp]
    \centering
\resizebox{\textwidth}{!}{
\begin{tabular}{c@{}c@{}c@{}c@{}l}
        \includegraphics[width = 0.24\textwidth]{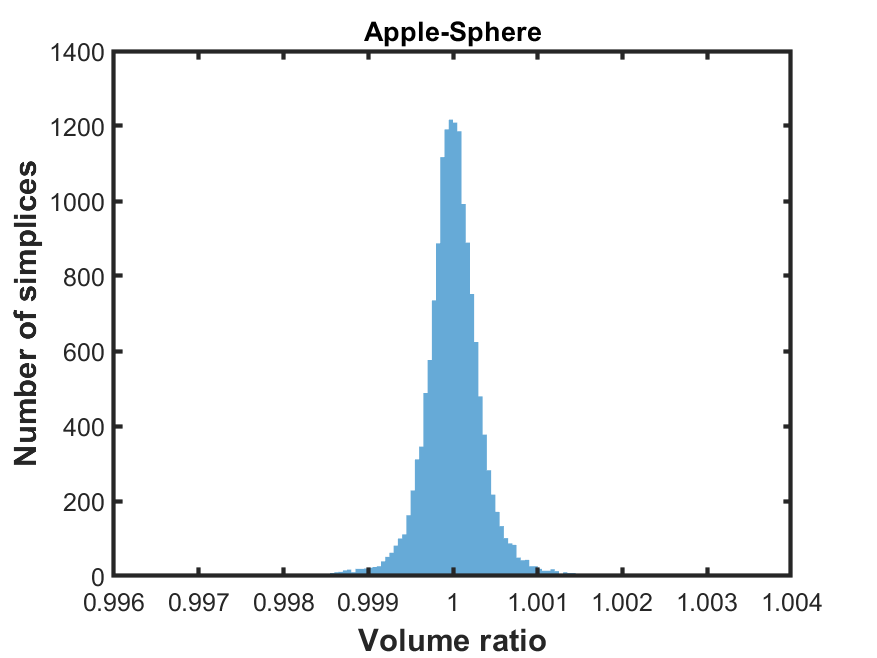} &
        \includegraphics[width = 0.24\textwidth]{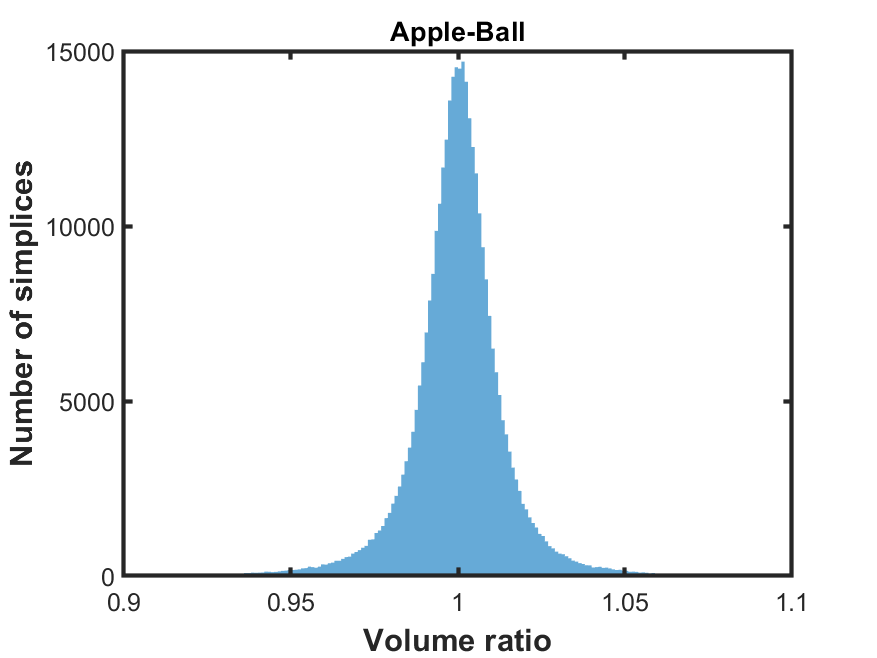} &
        \includegraphics[width = 0.24\textwidth]{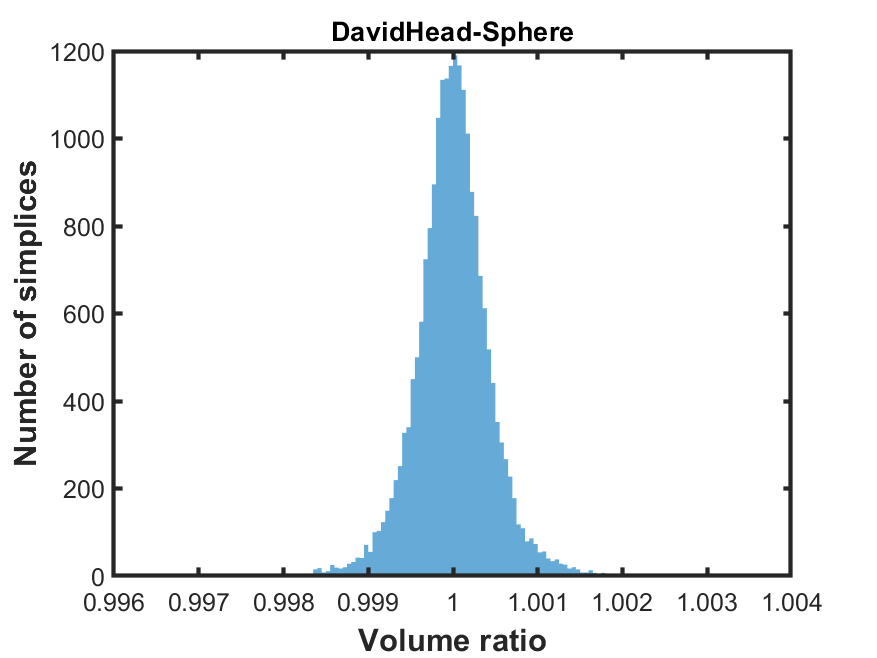} &
        \includegraphics[width = 0.24\textwidth]{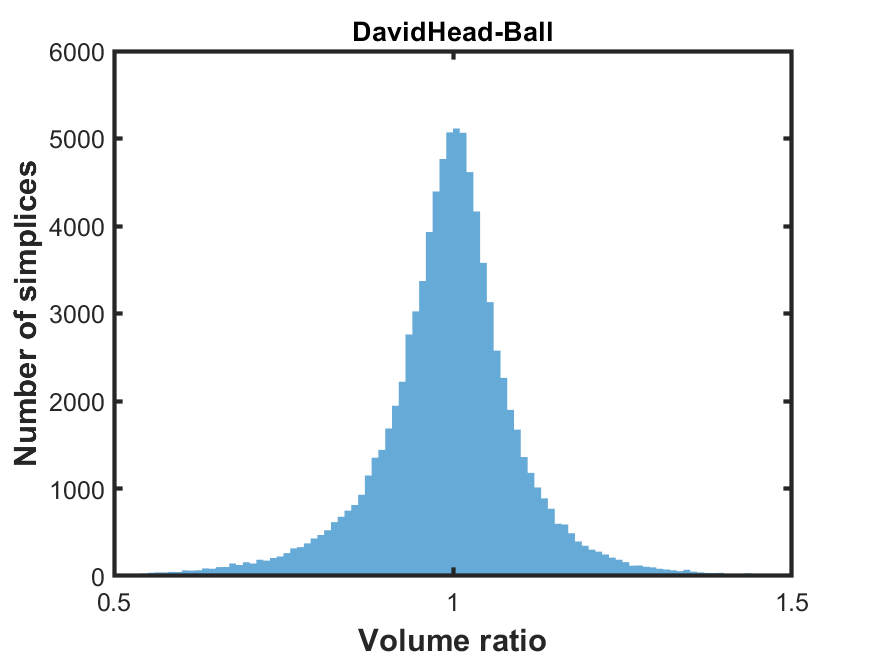} \\
        \includegraphics[width = 0.24\textwidth]{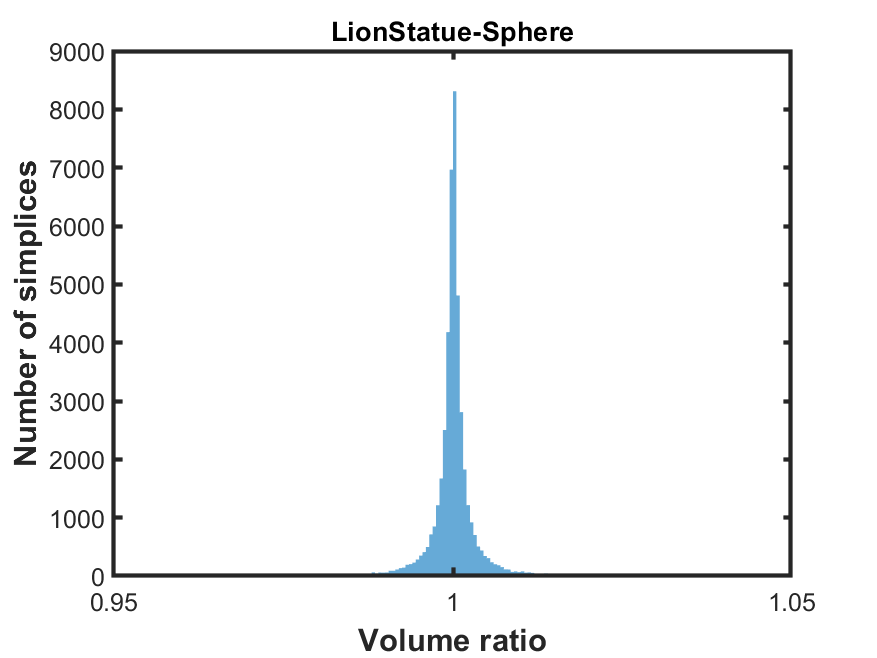} &
        \includegraphics[width = 0.24\textwidth]{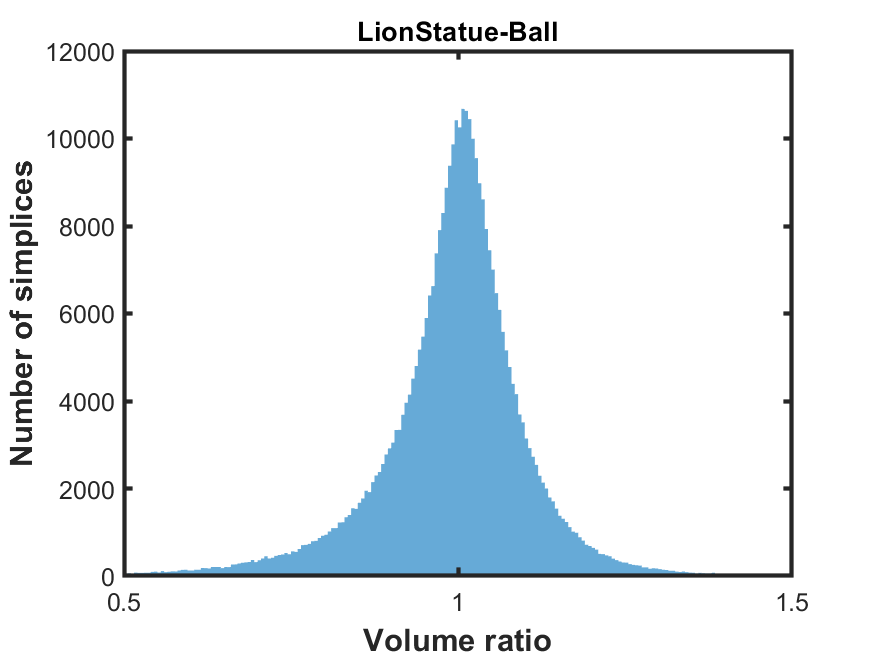} &
        \includegraphics[width = 0.24\textwidth]{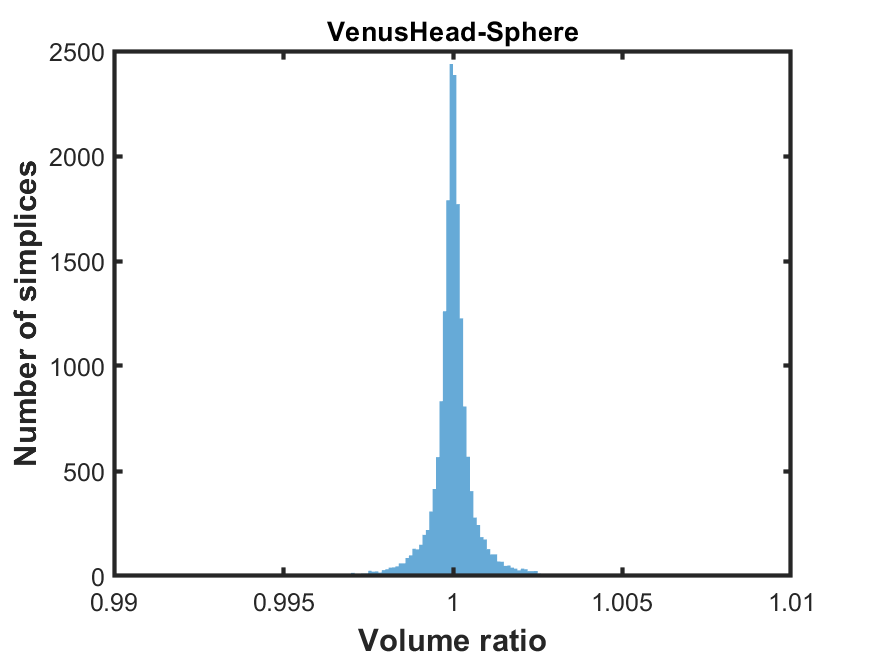} &
        \includegraphics[width = 0.24\textwidth]{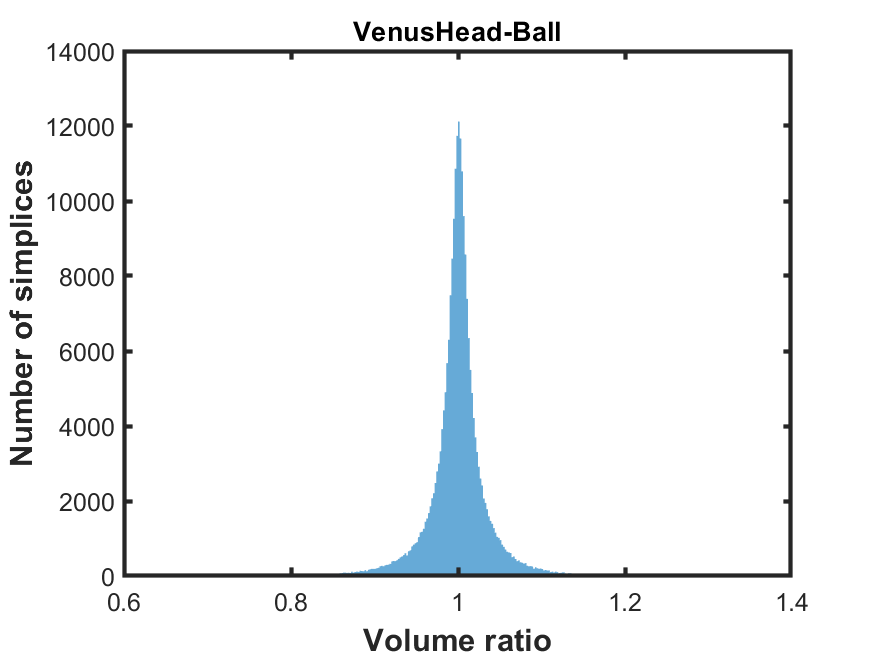}\\
        \includegraphics[width = 0.24\textwidth]{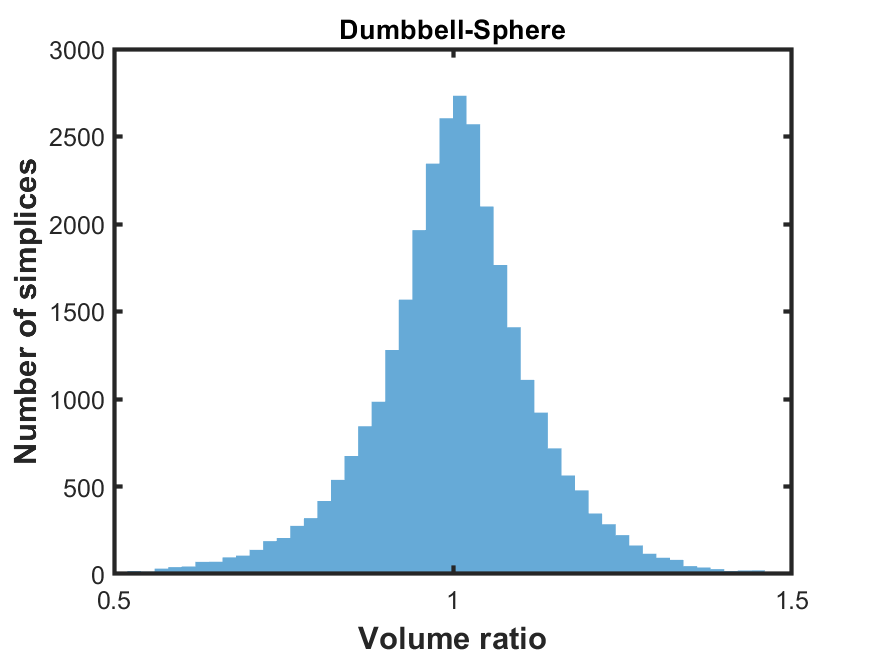} &
        &
        \includegraphics[width = 0.24\textwidth]{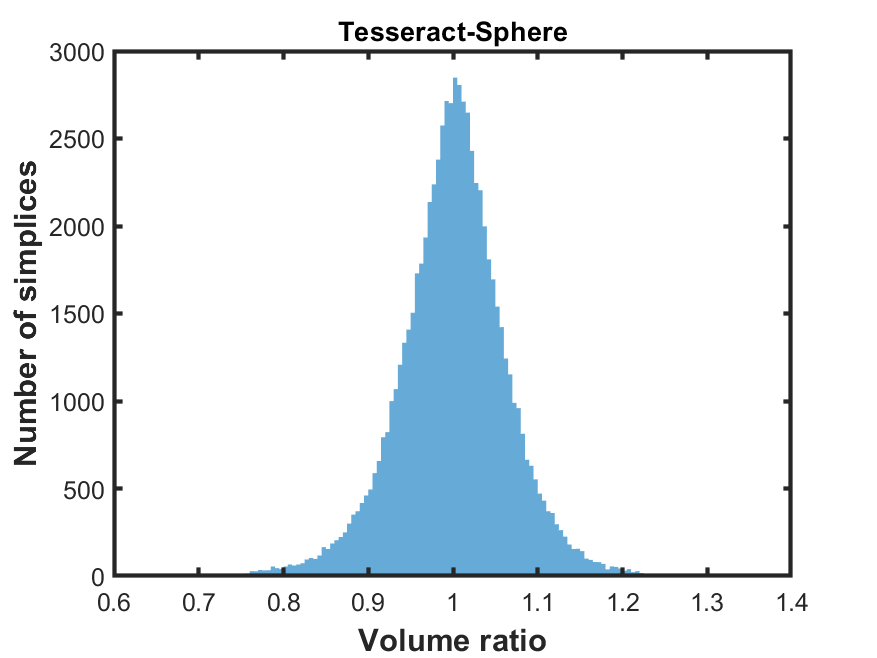} &
        \includegraphics[width = 0.24\textwidth]{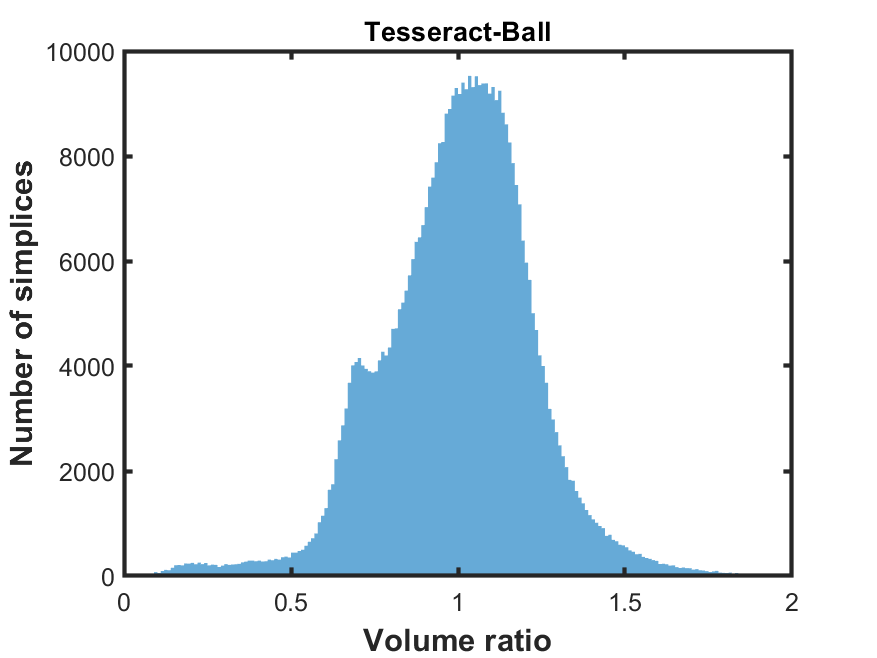}
\end{tabular}
}
\caption{Histograms of volume ratios $\delta + 1$ on the benchmarks.}
    \label{fig:mesh}
\end{figure}

From the \Cref{tab:generalmesh}, we can see that the proposed spherical and ball algorithms have well performances on the volume-preservation among all the benchmarks. For the boundary problem, the measurements $\varepsilon$ are on the order $10^{-6}$ to $10^{-4}$ and the local relative error delta are from $10^{-7}$ to $10^{-5}$. For the interior problem, the measurements $\varepsilon$ are on the order $10^{-4}$ to $10^{-2}$ and the local relative error delta are $10^{-5}$ to $10^{-3}$. For the 4 dimensional benchmarks, the measurements are slightly larger. Improving the accuracy of higher dimensional parameterizations is still worthy and significant work in the future. Additionally, the histograms in \Cref{fig:Histogram} specifically show the distributions of the local volume ratios $\delta_i+1=\frac{\nicefrac{|f(\sigma_i)|}{\sum_{\sigma \in \mathbb{S}_n(\mathcal{M})} |f(\sigma_i)|}}{\nicefrac{\mu(\sigma_i)}{\sum_{\sigma \in \mathbb{S}_n(\mathcal{M})} \mu(\sigma)}}$. We can see that the ratios concentrate in 1 tightly, indicating the high volume-preservation of our proposed algorithms.

\section{Applications}

\subsection{Registration}

\subsection{Deformation}

\section{Conclusion}

In this paper, we propose the $n$-VSE functional to achieve the mass-preserving parameterization for the $n$-manifold topologically equivalent to the $n$-ball. The $n$-VSE can be treated as the sum of the mass ratios among the manifold. We give a lower bound of the $n$-VSE and proof that the $n$-VSE reaches the lower bound if and only if the map is mass-preserving. This is the foundation of our proposed algorithms that aim to minimize the $n$-VSE. In the discrete case, the $n$-VSE is equivalent to the energies defined in \cite{MHWW19,MHTL19,TMWH23} when $n = 2$ and $n = 3$. The discrete $n$-VSE can be formulated as the quadratic form with respect to the Laplacian matrix and the gradient of the discrete $n$-VSE is of the simple formulation, that is, the product of the Laplacian matrix and the image vertices. This property makes the $n$-VSEM more practical. To compute the mass-preserving parameterization of the $n$-manifold, we solve the boundary and interior subproblems by developing the spherical and the interior algorithms, respectively. The numerical experiment shows the high accuracy of our developed algorithms, especially for the $3$-manifolds.

However, the efficiency and accuracy, especially for the higher dimensional manifolds, are not fully satisfactory. The high dimensional manifolds need more vertices and simplices to represent with high resolution, which makes the scale of the problem more huge. It is still a challenging problem.

\bibliographystyle{plain}

%\bibliographystyle{abbrv}
%\bibliography{reference}

%\printbibliography
\end{document}